\documentclass[11pt,reqno,a4paper]{amsart}

\usepackage{esint}
\usepackage{MnSymbol}
\usepackage{enumitem}
\usepackage[english]{babel}
\usepackage[utf8x]{inputenc}

\usepackage{tikz}
\usetikzlibrary{calc, shapes, backgrounds, patterns}

\usepackage[
	colorlinks,
	pdfpagelabels,
	pdfstartview = FitH,
	bookmarksopen = true,
	bookmarksnumbered = true,
	linkcolor = blue,
	plainpages = false,
	hypertexnames = false,
	citecolor = red] {hyperref}

\hypersetup{
	linktoc=page
}

\usepackage[capitalize]{cleveref}
\crefname{enumi}{}{}
\crefname{equation}{}{}

\makeatletter
\def\@tocline#1#2#3#4#5#6#7{\relax
  \ifnum #1>\c@tocdepth % then omit
  \else
    \par \addpenalty\@secpenalty\addvspace{#2}%
    \begingroup \hyphenpenalty\@M
    \@ifempty{#4}{%
      \@tempdima\csname r@tocindent\number#1\endcsname\relax
    }{%
      \@tempdima#4\relax
    }%
    \parindent\z@ \leftskip#3\relax \advance\leftskip\@tempdima\relax
    \rightskip\@pnumwidth plus4em \parfillskip-\@pnumwidth
    #5\leavevmode\hskip-\@tempdima
      \ifcase #1
       \or\or \hskip 1em \or \hskip 2em \else \hskip 3em \fi%
      #6\nobreak\relax
    \dotfill\hbox to\@pnumwidth{\@tocpagenum{#7}}\par
    \nobreak
    \endgroup
  \fi}
\makeatother

\usepackage{todonotes}

\newtheorem{theorem}{Theorem}[section]

\newtheorem{lemma}[theorem]{Lemma}

\newtheorem{conjecture}[theorem]{Conjecture}
\theoremstyle{definition}
\newtheorem{definition}[theorem]{Definition}
\newtheorem{remark}[theorem]{Remark}

\numberwithin{equation}{section}

\def \R {{\mathbb {R}}}

\def\supp{\operatorname{supp}}

\def\grad{\nabla}
\renewcommand{\tilde}{\widetilde}
\newcommand{\bigset}[1]{
	\big \{ #1 \big \}
}
\newcommand{\biggset}[1]{
	\bigg \{ #1 \bigg \}
}
\newcommand{\set}[1]{
	\left \{ #1 \right \}
}

\newcommand{\dx}{\, {\rm d} x}
\newcommand{\dt}{\, {\rm d} t}
\newcommand{\de}{\, {\rm d}}

\newcommand{\ddeg}{\mathcal{D}_{\mathcal{N}}}
\newcommand{\dpar}{\mathcal{D}}
\newcommand{\ddegg}{\mathcal{D}_{\mathcal{T}}}

\allowdisplaybreaks

\begin{document}
	
%%%%%%%%%%%%%%%%%%%%%%%%%%%%%%%%%%%%%%%%%%%%%%%%%%%%%%%%%%
% Title:author and so on 
%%%%%%%%%%%%%%%%%%%%%%%%%%%%%%%%%%%%%%%%%%%%%%%%%%%%%%%%%%

\title[Boundary Behavior ... II]{Boundary behavior of solutions to the parabolic $p$-Laplace equation II}

\author[Avelin]{Benny Avelin}
\address{
	Benny Avelin,
	Department of Mathematics, 
	Uppsala University,
	S-751 06 Uppsala, 
	Sweden
} 
\email{\color{blue} benny.avelin@math.uu.se}

\date{\today}



\begin{abstract}
	\noindent This paper is the second installment in a series of papers concerning the boundary behavior of solutions to the $p$-parabolic equations. 
	In this paper we are interested in the short time behavior of the solutions, which is in contrast with much of the literature, where all results require a waiting time. We prove a dichotomy about the decay-rate of non-negative solutions vanishing on the lateral boundary in a cylindrical $C^{1,1}$ domain. Furthermore we connect this dichotomy to the support of the boundary type Riesz measure related to the $p$-parabolic equation in NTA-domains, which has consequences for the continuation of solutions.
\end{abstract}
\maketitle

%%%%%%%%%%%%%%%%%%%%%%%%%%%%%%%%%%%%%%%%%%%%%%%%%%%%%%%%%%
% Paper starts 
%%%%%%%%%%%%%%%%%%%%%%%%%%%%%%%%%%%%%%%%%%%%%%%%%%%%%%%%%%

\section{Introduction} % (fold)
Recently there has been an upsurge in progress concerning the boundary behavior of solutions to the $p$-parabolic equation \cite{A,AGS,AKN,BBG,BBGP,GNL, GS, KMN}. 
%Recently, in \cite{AKN}, we developed some very flexible and intricate Harnack chain estimates and used them to prove the Carleson estimate, the boundary Harnack inequality, local and global. However the underlying assumption was that the intrinsic waiting time was smaller than the time of existence, hence we could not provide any insight into the behavior for time-scales shorter than the intrinsic waiting time. This is the aim of the current paper, that is to study what happens near the boundary for short time-scales and connecting this to the behavior on longer time-scales, in a sense bridging the short time behavior and the results in \cite{AKN}.

Before our paper \cite{AGS} the behavior of solutions close to the boundary was largely unknown in excess of regularity estimates, in \cite{AGS} we proved a Carleson estimate for Lipschitz domains and also pointed out that in the degenerate regime ($p > 2$), the failure of forward Harnack chains stems from the possible decay of certain solutions near the boundary, bounded by $d(x,\partial \Omega)^{q}$ for $q \geq \frac{p}{p-2}$. In \cite{AGS} we also noted that the boundary Harnack inequality cannot be true in the same form as for $p=2$, however in \cite{AKN} we proved that if the solutions have existed for long enough then the boundary Harnack inequality holds. This was a major breakthrough, but understanding of how badly the boundary Harnack can fail for shorter time-scales was still lacking. The purpose of this paper is to bridge this gap by proving exactly how badly the boundary Harnack inequality fails. We do this by studying the decay-rate of solutions that vanish on the boundary and discover a certain decay-rate phenomenon, a dichotomy, that only occurs in the degenerate regime. To describe this phenomenon we first state our equation

%Furthermore in \cite{AKN} we showed that even though

%{\color{red} In \cite{AKN} we where dealing with the boundary behavior with intrisic waiting times, the hidden supposition in that paper is that our assumptions imply that the solution lives for a long enough time to rule out certain solutions that only exist for a short time. Therefore we need other methods to figure out what happens in short time scales, i.e. what is the real need for the intrinsic waiting time? For short time scales the behavior of the p-parabolic equation deviates substantially from the linear $(p=2)$ case, and as such we need new ideas. This paper relies on iterated use of several barriers, some well known and some knew ones, but also relies on the full force of the Harnack chain analysis that we did in \cref{AKN}.}
%\todo{Rewrite the introduction to connect more towards the first paper}Recently there has been an upsurge in progress concerning the boundary behavior of solutions to the $p$-parabolic equation \cite{A,AGS,AKN,BBG,BBGP,GNL, GS, KMN}. Building upon this line of work, specifically as a direct descendant of \cite{AKN} (hence the title), this work will be concerned with the `short time' boundary behavior of solutions to the $p$-parabolic equation.\todo{Here we mention the short time, so a discussion could naturally fit here}

%The purpose of this paper is a ``proof of concept'' for a certain decay-rate phenomenon that only occurs in the degenerate regime ($p > 2$).\todo{Let us call it something else than a proof of concept.} To describe this phenomenon let us first state our equation
\begin{equation} \label{ppar1}
	u_t - \Delta_p u = u_t - \nabla \cdot (|Du|^{p-2} Du) = 0, \quad p > 2,
\end{equation}
we call \Cref{ppar1} the degenerate $p$-parabolic equation. 
The phenomenon consists of two related observations, first of all is the existence of solutions with different decay rates close to the boundary, let us illustrate with an example.
%, secondly it concerns the fact that solutions with initial data with fast decay preserves that behavior for some time (memory effect). We will now highlight these two observations with explicit examples, starting the existence of solutions with different decay rate
%The phenomenon concerns the behavior of solutions vanishing on relatively smooth boundary with respect to the decay-rate up to the boundary. The main example of this phenomenon is the following, 
Let $T \in \R$, then in $\R_+^n \times (-\infty,T)$ consider
\begin{align} \label{examp1}
	u_1(x,t)&=C(p)(T-t)^{-\frac{1}{p-2}%
	}x_{n}^{\frac p{p-2}}.
\end{align}
Another solution which we simply obtain is $u_2 = x_n$, the point is that $u_1$ and $u_2$ behave very differently at the boundary and shows, for instance, that a short time boundary Harnack inequality cannot hold. That is the ratio
\begin{align*}
	\frac{u_2}{u_1} \qquad \text{is not bounded from above and below.}
\end{align*} 
%However herein lies the interesting point, is there anything in-between these two solutions in terms of decay-rate?
Another observation concerning the solution in \cref{examp1} is that 
%In \cite{AGS} we pointed out that the solution $u_1$ in \Cref{examp1} is 
the solution is exactly so small that one cannot build Harnack chains up to the boundary, as noted in \cite{AGS}.
%, see \cref{cons:harnack}. 
In fact for the $p$-parabolic equation, the waiting time for $u_1$ dictated by the Harnack inequality at a point $(x_0,t_0)$ is
\begin{equation} \nonumber
	C (T-t_0)x_0^{-p}r^p.
\end{equation}
This implies that if we wish to apply the Harnack inequality with a radius comparable to the distance to the boundary then the waiting time will be of the same order of magnitude as the distance from $t_0$ to $T$ (the end of existence), in essence this implies that a Harnack chain cannot be performed.

%Solutions as in \Cref{examp1} are inherent to non-homogeneous parabolic equations (see \cref{sec:memory}). 

% Looking at the example in \Cref{examp1} it stands to reason that the ``rate of decay'' $d(x)^{\frac{p}{p-2}}$ is a sort of threshold. 
%It can be argued that the $p$-parabolic equation has a certain ``memory'' concerning the decay-rate of the initial data. We highlight this idea with another example from \cite{AGS}, namely
The other aspect is that solutions retain some of the behavior of the initial data in the sense described by the following example from \cite{AGS}. That is, consider the following supersolution to \Cref{ppar1} 
\begin{equation}\nonumber %  \label{examp2}
	u(x,t)=C(p)(T-t)^{-\frac1{p-2}}\exp(-\frac1{x}),\qquad x\in (0,1/4).
\end{equation}
It is now clear that if our domain is $E_T = (0,1) \times (0,T]$ then the Cauchy-Dirichlet problem in $E_T$ has a short time behavior that is heavily dictated by the decay of the initial data, in contrast to the linear case, where no matter what positive data we have we will always have that $\partial_x u(0,t) > 0$ for all $t \in (0,T)$. We explore this ``memory effect'' further in \cref{sec:memory}.

Our prime motivating idea is that the existence of a forward Harnack chain should imply that the equation behaves like a non-degenerate equation, in the sense that there would be no memory effect and there can only exist one behavior at the boundary, i.e.~solutions vanish like the distance function. We can state this simply as, there should be no solution that vanishes slower than $u_1$ in \cref{examp1} and at the same time, faster than $u_2=x_n$. To state this precisely we need to introduce some concepts, we start with the definition of a degenerate boundary point

%To describe our results regarding this decay-rate phenomenon we first need some definitions, we begin with what we term a degenerate point, i.e. 

\begin{definition} \label{def:deg}
	Let $\Omega \subset \R^n$ be a domain and let $u$ be a solution to \cref{ppar1} in $\Omega_T = \Omega \times (0,T)$. We call $w \in \partial \Omega$ a degenerate point at $t_0 \in (0,T)$ with respect to $u$ if
	\begin{align*}
		\mathcal{N}^+\left [ \frac{u(x,t_0)}{|x-w|^{\frac{p}{p-2}}} \right ] (w) < \infty,
	\end{align*}
	where
	\begin{align*}
		\mathcal{N}^+ [f] (w) &= \limsup_{x \in \Gamma (w), x \to w} f(x) \\
		\Gamma_\alpha(w) &= \set{x \in \Omega: \alpha|x-w| < d(x,\partial \Omega)}
	\end{align*}
	for some $\alpha \in (0,1)$.
\end{definition}
\begin{remark} \label{rem:Non-tangential}
	The value of $\alpha > 0$ in the above definition does not enter into any of the estimates and is irrelevant as long as $\Gamma(w)$ satisfies
	\begin{align} \label{eq:non-empty-NTA}
		\Gamma(w) \cap B_r(w) \neq \emptyset, \quad  \forall r > 0.
	\end{align}
	Since we will be working in $(M,r_0)$-NTA-domains (see \cref{NTA}), we immediately see that if $\alpha < 1/M$ the above property is true, since 
	\begin{align*}
		a_\varrho(w) \in \Gamma(w), \quad \forall \varrho < r_0.
	\end{align*}
	Hence in the rest of this paper we will ignore the value of $\alpha > 0$, but we will always assume that $\alpha > 0$ is chosen such that \cref{eq:non-empty-NTA} holds, which according to the above means that in $(M,r_0)$-NTA-domains we assume $\alpha < 1/M$.
\end{remark}

Secondly we define what we mean by a non-degenerate point, i.e.

\begin{definition}
	Let $\Omega \subset \R^n$ be a domain and let $u$ be a solution to \cref{ppar1} in $\Omega_T = \Omega \times (0,T)$. We call $w \in \partial \Omega$ a non-degenerate point at $t_0 \in (0,T)$ with respect to $u$ if
	\begin{align*}
		\limsup_{\Omega \ni x \to w} \frac{d(x,\partial \Omega)}{u(x,t_0)} < \infty.
	\end{align*}
\end{definition}

The main result in this paper is that for NTA-domains that satisfy the interior ball condition there can only be degenerate and non-degenerate points $(w,t)$ for a given $w \in \partial \Omega$ except for possibly a single time $\hat t$ that we call the threshold point, see Figure \ref{fig:first}. This closes the gap between short-time and long-time boundary behavior in the degenerate regime and gives together with \cite{AKN} a full picture in smooth domains. As mentioned before the existence of degenerate points in the sense of \cref{def:deg} is purely a non-linear fact and does not occur for $p>2$. We prove the main result using a sequence of fairly convoluted comparisons with carefully selected barrier-/comparison-functions, also relying on the full force of the delicate estimates developed in \cite{AKN}.

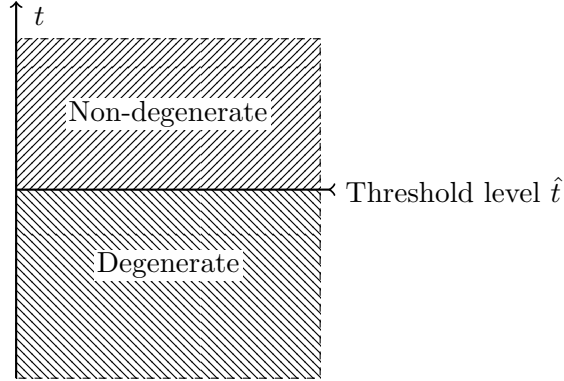
\begin{figure}
	\label{fig:first}
	\begin{center}
		\begin{tikzpicture}
			\draw [thick,->] (0,0)--(0,5);
			\node[right] at (0.1,4.8) {$t$};
			\draw[pattern=north east lines,dashed] (0,2.5) rectangle (4,4.5);
			\node[fill=white,inner sep=1pt] at (2,3.5) {Non-degenerate};
			\draw[thick,-<] (0,2.5)--(4.2,2.5);
			\node[right] at (4.2,2.5) {Threshold level $\hat t$};
			\draw[pattern=north west lines,dashed] (0,0) rectangle (4,2.5);
			\node[fill=white,inner sep=1pt] at (2,1.5) {Degenerate};
		\end{tikzpicture}
	\end{center}
	\caption{Schematic overview of basic diffusion}
\end{figure}

\subsection{Outline of paper}
We begin the contents of the paper in \cref{sec:Definitions} and \cref{sec:preliminary}, where we provide all the definitions and results needed for the bulk of the paper. Next in \cref{sec:main results} we state all the main results and prove some of the simple but powerful consequences. The rest of the sections is devoted to proofs of the main results, except for \cref{sec:cone,sec:interface}. In \cref{sec:cone} we give an example of a solution with support never reaching the boundary in a conical domain. Finally in \cref{sec:interface} we deal with an example having stationary support for a non-zero time interval, we theorize about the length of that interval and provide some numerical computations concerning its length.

\vspace{0.5cm}

\noindent {\bf Acknowledgment} The author was supported by the Swedish Research Council, dnr: 637-2014-6822.

\section{Definitions and notation}
\label{sec:Definitions}
Points in $ \R^{n+1} $ are denoted by $ x = ( x_1, \dots, x_n,t)$. Given a set $E\subset \R^n$, let $ \bar E, \partial E$, $\mbox{diam } E$, $E^c$, $E^\circ $, denote the closure, boundary, diameter, complement and interior of $E$, respectively. Let $ \cdot $ denote the standard inner product on $ \R^{n} $, let $ | x | = (x \cdot x )^{1/2}$ be the Euclidean norm of $ x, $ and let $\dx $ be Lebesgue $n$-measure on $ \R^{n}. $ Given $ x \in \R^{n}$ and $r >0$, let $ B_{r} (x) = \{ y \in \R^{n} : | x - y | < r \}$. Given $ E, F \subset \R^{n}, $ let $ d ( E, F ) $ be the Euclidean distance from $ E $ to $ F $. In case $ E = \{y\}, $ we write $ d ( y, F )$. For simplicity, we define $\sup$ to be the essential supremum and $\inf$ to be the essential infimum. If $ O \subset \R^{n} $ is open and $ 1 \leq q \leq \infty, $ then by $ W^{1 ,q} ( O )$ we denote the space of equivalence classes of functions $ f $ with distributional gradient $ \nabla f = ( f_{x_1}, \dots, f_{x_n} ), $ both of which are $ q $-th power integrable on $ O. $ Let 
\begin{align*}
		\| f \|_{ W^{1,q} (O)} = \| f \|_{ L^q (O)} + \| \, | \nabla f | \, \|_{ L^q ( O )}
\end{align*}
be the norm in $ W^{1,q} ( O ) $ where $ \| \cdot \|_{L^q ( O )} $ denotes the usual Lebesgue $q$-norm in $O$. $ C^\infty_0 (O )$ is the set of infinitely differentiable functions with compact support in $ O$ and we let $ W^{1 ,q}_0 ( O )$ denote the closure of $ C^\infty_0 (O )$ in the norm $\| \cdot\|_{ W^{1,q} (O)}$. $ W^{1,q}_{\rm loc} ( O ) $ is defined in the standard way. By $ \nabla \cdot $ we denote the divergence operator. Given $t_1<t_2$ we denote by $L^q(t_1,t_2,W^{1,q} ( O ))$ the space of functions such that for almost every $t$, $t_1\leq t\leq t_2$, the function $x\to u(x,t)$ belongs to $W^{1,q} ( O )$ and
\begin{equation*}
	\| u \|_{ L^q(t_1,t_2,W^{1,q} ( O ))}:=\biggl (\int\limits_{t_1}^{t_2}\int\limits_O\biggl (|u(x,t)|^q+|\nabla u(x,t)|^q\biggr )\dx\dt\biggr )^{1/q} <\infty.
\end{equation*}
The spaces $L^q(t_1,t_2,W^{1,q}_0 ( O ))$ and $L^q_{\rm loc}(t_1,t_2,W^{1,q}_{\rm loc} ( O ))$ are defined analogously. Finally, for $I \subset \R$, we denote $C(I;L^{q} ( O ))$ as the space of functions such that $t\to \| u(t,\cdot) \|_{ L^{q} (O)}$ is continuous whenever $t \in I$. $C_{\rm loc}(I;L^{q}_{\rm loc} ( O ))$ is defined analogously.

%%%%%%%%%%%%%%%%%%%%%%%%%%%%%%%%%%%%%%%%%%%%%%%%%%%%%%%%%%

\subsection{Weak solutions} Let $ \Omega \subset \R^n $ be a bounded domain, i.e., a connected open set. For $t_1<t_2$, we let $\Omega_{t_1,t_2}:= \Omega \times (t_1,t_2)$. Given $p$, $1<p<\infty$, we say that $ u $ is a weak solution to
\begin{equation} \label{Hu}
	\partial_t u - \Delta_p u = 0
\end{equation}
in $ \Omega_{t_1,t_2}$ if $ u \in L_{\rm loc}^p(t_1,t_2,W_{\rm loc}^{1,p} ( \Omega))$ and
\begin{equation}
	\label{1.1tr} \int_{\Omega_{t_1,t_2} } \left(- u \partial_t \phi  + |\nabla u|^{p-2}\nabla u\cdot\nabla \phi \right) \dx\dt = 0
\end{equation}
whenever $ \phi \in C_0^\infty(\Omega_{t_1,t_2})$. First and foremost we will refer to equation \cref{Hu} as the $p$-parabolic equation and if $ u $ is a weak solution to~\Cref{Hu} in the above sense, then we will often refer to $u$ as being $p$-parabolic in $\Omega_{t_1,t_2}$. For $p \in (2,\infty)$ we have by the parabolic regularity theory, see \cite{DB}, that any $p$-parabolic function $u$ has a locally H{\"o}lder continuous representative. In particular, in the following we will assume that $p \in (2,\infty)$ and any solution $u$ is continuous. If \Cref{1.1tr} holds with $=$ replaced by $\geq$ ($\leq$) for all $ \phi \in C_0^\infty(\Omega_{t_1,t_2})$, $\phi \geq 0$, then we will refer to $u$ as a weak supersolution (subsolution).

%%%%%%%%%%%%%%%%%%%%%%%%%%%%%%%%%%%%%%%%%%%%%%%%%%%%%%%%%%

\subsection{Geometry} We here state the geometrical notions used throughout the paper.
\begin{definition}
	\label{NTA} A bounded domain $\Omega$ is called non-tangentially accessible (NTA) if there exist $M \geq 2$ and $r_0$ such that the following are fulfilled:
	\begin{enumerate}[label=\bf{(\arabic*)}]
		\item \label{NTA1} \emph{corkscrew condition:} for any $ w\in \partial\Omega, 0<r<r_0,$ there exists a point $a_r(w) \in \Omega $ such that
		\begin{equation}
			\nonumber M^{-1}r<|a_r(w)-w|<r, \quad d(a_r(w), \partial\Omega)>M^{-1}r ,
		\end{equation}
		\item \label{NTA2} $\R^n \setminus \Omega$ satisfies \cref{NTA1},
		\item \label{NTA3} \emph{uniform condition:} if $ w \in \partial \Omega, 0 < r < r_0, $ and $ w_1, w_2 \in B_r(w) \cap \Omega, $ then there exists a rectifiable curve $ \gamma: [0, 1] \to \Omega $ with $ \gamma ( 0 ) = w_1, \gamma ( 1 ) = w_2, $ such that
		\begin{enumerate}
			\item $H^1 ( \gamma ) \, \leq \, M \, | w_1 - w_2 |,$
			\item $\min\{H^1(\gamma([0,t])), \, H^1(\gamma([t,1]))\, \}\, \leq \, M \, d ( \gamma(t), \partial \Omega)$, for all $t \in [0,1]$.
		\end{enumerate}
	\end{enumerate}
\end{definition}

The values $ M $ and $r_0$ will be called the NTA-constants of $ \Omega$. For more on the notion of NTA-domains we refer to \cite{JK}.

\begin{definition}
	\label{defBall}
	Let $\Omega \subset \R^n$ be a bounded domain. We say that $\Omega$ satisfies the interior ball condition with radius $r_0 > 0$ if for each point $y \in \partial \Omega$ there exist a point $x^+ \in \Omega$ such that $B_{r_0}(x^+) \subset \Omega$ and $\partial B_{r_0}(x^+) \cap \partial \Omega =\{y\}$.
\end{definition}

%%%%%%%%%%%%%%%%%%%%%%%%%%%%%%%%%%%%%%%%%%%%%%%%%%%%%%%%%%

\subsection{The continuous Dirichlet problem} Assuming that $\Omega$ is a bounded NTA-domain one can prove, see \cite{BBGP} and \cite{KL}, that all points on the parabolic boundary
\begin{equation*}
	\partial_p\Omega_T = S_T \cup (\bar \Omega \times \{0\})\,, \qquad S_T = \partial \Omega \times [0,T],
\end{equation*}
of the cylinder $\Omega_T$ are regular for the Dirichlet problem for equation \Cref{Hu}. In particular, for any $f\in C( \partial_p\Omega_T)$, there exists a unique Perron-solution $u=u_f^{\Omega_T}\in C(\overline \Omega_T)$ to the Dirichlet problem \cref{Hu} in $\Omega_T$ and $u =f$ on $\partial_p \Omega_T$.

In the study of the boundary behavior of quasi-linear equations of $p$-Laplace type, certain Riesz measures supported on the boundary and associated to non-negative solutions vanishing on a portion of the boundary are important, see \cite{LN1,LN2}. These measures are non-linear generalizations of the harmonic measure relevant in the study of harmonic functions. Corresponding  measures can also be associated to solutions to the $p$-parabolic equation. Let $u$ be a non-negative solution in $\Omega_T$, assume that $u$ is continuous on the closure of $\Omega_T$, 
and that $u$ vanishes on $\partial_p \Omega_T \cap Q$ with some open set $Q$. Extending $u$ to be zero in $Q \setminus \Omega_T$, we see that $u$ is a continuous weak subsolution to~\cref{Hu} in $Q$. From this one sees that there exists a unique locally finite positive Borel measure $ \mu$, supported on $S_T \cap  Q$, such that
\begin{align}
	\label{eq:Riesz} & -\int_Q u \partial_t\phi\dx \dt +\int_Q|\nabla u|^{p-2}\nabla u\cdot\nabla \phi \dx \dt = -\int_Q \phi \de \mu
\end{align}
whenever $ \phi \in C_0^\infty(Q)$. Whenever we have a solution and when there is no danger of confusion we will simply use $\mu$ to denote the corresponding measure, in other cases we will subscript the measure with the solution, i.e. for a solution $v$ we will use the notation $\mu_v$.

% Define the non-tangential limsup as follows, let
% \begin{align*}
% 	\Gamma_\alpha(w) = \set{x \in \Omega: \alpha|x-w| < d(x,\partial \Omega)}
% \end{align*}
% where $\alpha \in (0,1)$ is irrelevant to the qualitative estimates we will be concerned with. Thus we can define a limit as follows (suppressing the dependence on $\alpha$)
% \begin{equation} \nonumber
% 	\mathcal{N}^+ [f] (w) = \limsup_{x \in \Gamma (w), x \to w} f(x)\,.
% \end{equation}

\section{Preliminary estimates: Carleson and Backward Harnack chains}
\label{sec:preliminary}
The proofs of this paper relies on the following estimates from \cite{AKN}, we we include for the ease of the reader.
The following estimate is a simple Harnack chain lemma for forward in time Harnack chain that we developed in \cite{AKN}.

\begin{lemma} \label{lem AKN HChain}
	Let $\Omega\subset\R^N$ be a domain and let $T>0$. Let $x,y$ be two points in $\Omega$ and assume that there exist a sequence of balls $\set{B_{4r}(x_j)}_{j=0}^k$ such that $x_0=x$, $x_k=y$, $B_{4r}(x_j)\subset\Omega$ for all $j=0,...,k$ and that $x_{j+1} \in B_{r}(x_j)$, $j=0,\ldots,k-1$. Let $u$ be a non-negative solution to \eqref{Hu} in $\Omega_T$ and assume that $u(x,t_0)>0$. There exist constants $\bar c_i \equiv \bar c_i(p,n)$, $i \in \set{1,2}$ and $\bar c_3\equiv \bar c_3(p,n,k)>1$ such that if 
	\begin{equation*}
		t_0 - (\bar c_1/u(x,t_0))^{p-2} (4r)^p > 0,\ t_0 + \bar c_3(k) u(x,t_0)^{2-p} r^p <T, 
	\end{equation*}
	then 
	\begin{equation*}
		u(x,t_0) \leq \bar c_2^{k} \inf_{z \in B_r(y)} u(z,t_0+\bar c_3(k) u(x,t_0)^{2-p} r^p). 
	\end{equation*}
\end{lemma}

As we already mentioned in \cite{AKN} there is a vast difference between Harnack chains performed backwards in time versus chains performed forward in time. This is a point of philosophical nature. When building Harnack chains forward in time we solely use the known information at the point of reference. In contrast, when performing backward Harnack chains we are instead considering the question, how large could the solution have been in the past such that the solution is below a given value at a given reference point. In essence backwards chains does not really rely on the values of the actual solution but a forward chain is forced to do so. This has the consequence that forward chains have a waiting time that we have no control over, while for backward chains we have a fairly fine control over the waiting time, actually since we can change the reference value $\Lambda$ we have more control over the waiting time than we have in the linear setting (see \cite{A}).

We will need the following version of the backward Harnack chain theorem that we proved in \cite{AKN}, this is an updated version of the results in \cite{AGS} with more control over the waiting time and also valid in NTA-domains.

\begin{theorem}
	\label{NTA BChain} Let $\Omega\subset\R^n$ be an NTA-domain with constants $M$ and $r_0$, let $x_0 \in \partial \Omega$, $T>0$, and let $0 < r < r_0$. Let $x,y$ be two points in $\Omega \cap B_r(x_0)$ such that 
	\begin{equation*}
		\varrho := d(x,\partial \Omega) \leq r \qquad \mbox{and} \qquad d(y,\partial \Omega) \geq \frac{r}{4}. 
	\end{equation*}
	Assume that $u$ is a non-negative solution to \cref{Hu} in $\Omega_T$, and assume that $\Lambda \geq u( y , s)$ is positive. Let $\delta \in (0,1]$. Then there exist positive constants $C_i \equiv C_i(p,n)$ and $c_i \equiv c_i(p,n,M)$, $i\in\{4,5\}$, such that if $s<T$ and 
	\begin{equation}\nonumber %\label{eq:NTA BChain t cond} 
		\max\left\{ \left(\frac{c_4^{1/\delta}}{c_h} \left( \frac{r}{\varrho} \right)^{c_5/\delta} \Lambda \right)^{2-p } (\delta \varrho)^p , s - \tau \right\}\, \leq\, t\, \leq\, s - \delta^{p-1} \tau, 
	\end{equation}
	with 
	\begin{equation*}
		\tau := C_4 \left[C_5 \Lambda \right]^{2-p} r^p 
	\end{equation*}
	then 
	\begin{equation*}
		u(x,t) \leq c_4^{1/\delta} \left( \frac{r}{\varrho} \right)^{c_5/\delta} \Lambda. 
	\end{equation*}
	Furthermore, constants $c_i,C_i$, $i \in \{4,5\}$, are stable as $p \to 2^+$. 
\end{theorem}

\begin{proof}
	Rescaling such that $u(y,s) \leq \Lambda = 1$, the proof follows verbatim as in \cite{AKN}.
\end{proof}

The above version can then be used to prove the following slightly modified version of the same theorem found in \cite{AKN}, this estimate is also an updated version of a similar statement found in \cite{AGS}. The difference is in the flexibility of it usage and the generality of its validity.

\begin{theorem}
	\label{thm Carleson} Let $\Omega \subset \R^n$ be an NTA-domain with constants $M$ and $r_0$. Let $u$ be a non-negative solution to \cref{Hu} in $\Omega_T$. Let $(x,t) \in S_T$ and $0 < r < r_0$. Assume that $\Lambda \geq u(a_r(x),t)$ is positive and let 
	\begin{equation*}
		\tau = \frac{C_4}{4} \left[C_5 \Lambda \right]^{2-p} r^p , 
	\end{equation*}
	where $C_4$ and $C_5$, both depending on $p,n$, are as in~\Cref{NTA BChain}. Assume that $t > (\delta_1^{p-1}+\delta_2^{p-1} + 2\delta_3^{p-1}) \tau$ for $0< \delta_1\leq \delta_3 \leq 1$, $\delta_2 \in (0,1)$ and that for a given $\lambda \geq 0$, the function $(u - \lambda)_+$ vanishes continuously on $S_T \cap B_r(x) \times (t-(\delta_1^{p-1} + \delta_2^{p-1}+\delta_3^{p-1}) \tau, t-\delta_1^{p-1} \tau)$ from $\Omega_T$. Then there exist constants $c_i \equiv c_i(M,p,n)$ $i \{6,7\}$, such that 
	\begin{equation*}
		\sup_{Q} u \leq \left( c_6 / \delta_3 \right)^{c_7/ \delta_1} \Lambda + \lambda, 
	\end{equation*}
	where $Q := B_r(x) \times ( t - (\delta_1^{p-2}+\delta_2^{p-1})\tau, t - \delta_1^{p-1} \tau)$. Furthermore, constants $c_i$, $i \in \{6,7\}$, are stable as $p \to 2^+$
\end{theorem}
\begin{proof}
	By scaling the function $u$ we can assume that $u(a_r(x),t) \leq \Lambda = 1$, and replacing $\lambda$ with its scaled version. The proof now follows verbatim as in \cite{AKN}.
\end{proof}

\begin{theorem}
	\label{NTA FChain} Let $\Omega\subset\R^n$ be an NTA-domain with constants $M$ and $r_0$, let $x_0 \in \partial \Omega$, $T>0$ and let $0 < r < r_0$. Let $x,y$ be two points in $\Omega \cap B_r(x_0)$ such that
	\begin{equation*}
		\varrho := d(x,\partial \Omega) \leq r \qquad \mbox{and} \qquad d(y,\partial \Omega) \geq \frac{r}{4}.
	\end{equation*}
	Assume that $u$ is a non-negative $p$-parabolic function in $\Omega_T$, and assume that $u(x,t_0)$ is positive. Let $\delta \in (0,1]$. Then there exist constants $c_i \equiv c_i(M,p,n)$, $i\in\{1,2,3\}$, such that if
	\begin{equation*}
		t_0 - (c_h/u(x,t_0))^{p-2} (\delta \varrho)^p > 0, \qquad t_0 + \tau < T,
	\end{equation*}
	with
	\begin{equation*}
		\tau := \delta^{p-1} \left( c_2^{-1/\delta} \left( \frac{r}{\varrho} \right)^{- c_3/\delta} u(x,t_0) \right)^{2-p} r^p ,
	\end{equation*}
	then
	\begin{equation*}
		u(x,t_0) \leq c_1^{1/\delta} \left( \frac{r}{\varrho} \right)^{c_3/\delta} \inf_{z \in B_{r/16}(y)} u(z ,t_0 + \tau).
	\end{equation*}
	Furthermore, constants $c_i$, $i \in \{1,2,3\}$, are stable as $p \to 2^+$.
\end{theorem}

\section{Main results}
\label{sec:main results}
As alluded to in the introduction we will mainly be concerned with the split between the degenerate and non-degenerate boundary points, and the first step in this direction is the below result. It essentially states that as soon as a point $w$ is no longer degenerate at a time $\hat t$ the point is non-degenerate for the following times.

\begin{theorem}[Immediate linearization]\label{thm linearization}
	Let $\Omega \subset \R^N$ be a domain satisfying the interior ball condition with radius $r_0 > 0$. Let $u$ be a non-negative solution to \eqref{Hu} in $\Omega_T$ and assume that $w \in \partial \Omega$, $t_0 \in (0,T)$ and that the following holds
	\begin{equation} \label{eq linearization req}
		\mathcal{N}^+\left [ \frac{u(x,t_0)}{|x-w|^{\frac{p}{p-2}}} \right ] (w) = \infty.
	\end{equation}
	%for some function $\eta:\R_+ \to \R_+$ where $\eta(0^+) = \infty$. 
	Then for any $t_+ \in (t_0,T)$ we have
	\begin{equation} \label{eq linearized}
		\limsup_{\Omega \ni x \to w} \frac{d(x,\partial \Omega)}{u(x,t_+)} < \infty.
	\end{equation}
\end{theorem}

Now that we know that once the threshold has been reached then behavior changes, let us look at the next result which states that in the degenerate regime, i.e. smaller than $|x-w|^{\frac{p}{p-2}}$ then it continues to be smaller than $|x-w|^{\frac{p}{p-2}}$ for a small time interval, in essence it states that the degeneracy is an open condition.

\begin{theorem}[Local memory effect for degenerate initial data]\label{thm memory full}
	Consider a bounded domain $\Omega \subset \R^N$, assume that $w \in \partial \Omega$. Consider the domain $E_T := E \times (0,T) := (\Omega \cap B(w,r)) \times (0,T)$, assume that $u \in C(\overline{E_T})$ is a non-negative solution to \cref{Hu} vanishing on $(\partial \Omega \cap B(w,r)) \times [0,T)$, and that $u \leq M$ in $E_T$. If the initial data satisfies
	\begin{equation} \nonumber
		u_0(x) \leq M \frac{|x-w|^\delta}{r^\delta}, \quad \text{for a $\delta \geq \frac{p}{p-2}$,}
	\end{equation}
	then there exists a time
	\begin{equation*} 
		\hat T := \left [\frac{C(p,\delta)}{M} \right ]^{p-2} r^p
	\end{equation*}
	and a constant $c_0(p) > 0$ such that for $\tilde T := \min\{\hat T/2, T \}$ the following upper bound holds
	\begin{equation} \nonumber
		u(x,t) \leq c_0 M |x-w|^{\delta}, \quad (x,t) \in E_{\tilde T}.
	\end{equation}
\end{theorem}

\subsection{Classifying boundary points: a dichotomy}

In this section we take our results about the critical thresholds (\cref{thm linearization}) and use them to prove that there are only two different behaviors, i.e. we prove a simple dichotomy about the boundary points. Whats more we prove that they are ordered as intervals dividing the whole existence of a solution.

\begin{theorem} \label{thm dichotomy}
	Let $\Omega \subset \R^N$ be a domain satisfying the interior ball condition with radius $r_0 > 0$. Let $u$ be a non-negative solution to \cref{Hu} in $\Omega_T$. Let $w \in \partial \Omega$, and define the sets
	\begin{align*}
		\ddeg(w) &:= \biggset{t \in (0,T): \mathcal{N}^+\bigg [ \frac{u(x,t_0)}{|x-w|^{\frac{p}{p-2}}} \bigg ] (w) < \infty}, \\
		\dpar(w) &:= \biggset{t \in (0,T): \limsup_{\Omega \ni x \to w} \frac{d(x,\partial \Omega)}{u(x,t)} < \infty},
	\end{align*}
	then $\ddeg(w)$ and $\dpar(w)$ are disjoint intervals. If $\ddeg(w)$ and $\dpar(w)$ are both nonempty, then there exists a time $\hat t \in (0,T)$ such that
	\begin{equation} \notag %\label{eq 0Tsplit}
		(0,T) = (\ddeg(w))^\circ \cup \set{\hat t} \cup (\dpar(w))^\circ, \quad (0,\hat t) = (\ddeg(w))^\circ,
	\end{equation}
	and the union is disjoint.
	Otherwise if $\ddeg(w) = \emptyset$ or $\dpar(w) = \emptyset$ then
	\begin{equation} \nonumber
		(0,T) = \dpar(w) \cup \ddeg(w).
	\end{equation}
\end{theorem}

\begin{proof}
	Assume the first situation, i.e. that $\ddeg(w),\dpar(w) \neq \emptyset$.
	
	We first note that if $t \in \dpar(w)$ then from \Cref{thm linearization} we get that $[t,T) \subset \dpar(w)$, which implies that $\dpar(w)$ is a right-open interval. Thus $\dpar(w)$ can be written as either $(\hat t,T)$ or $[\hat t, T)$. From the assumption $\ddeg(w),\dpar(w) \neq \emptyset$ we have $\hat t \in (0,T)$.
	
	We wish to show that $(0,\hat t) \subset \ddeg(w)$ which implies that the set $\ddeg(w)$ is left-open. To do this, assume the contrary, i.e. that there exists a $t \in (0,\hat t)$ such that $t \not \in \ddeg(w)$. Applying \Cref{thm linearization} we get that $(t,T) \subset \dpar(w)$ contradicting that $\dpar(w)$ is of the form $(\hat t,T)$ or $[\hat t, T)$. It is now clear that $\ddeg(w)$ is also an interval, where $\hat t$ may or may not be included in $\ddeg(w)$.
	
	Lastly assume that we have the situation that $\ddeg(w) = \emptyset$, this implies that for any $t \in (0,T)$ we can apply \Cref{thm linearization} to get that $(t,T) \subset \dpar(w)$. Since $t$ was arbitrary we have that $(0,T) = \dpar(w)$. In a similar way we get that if $\dpar(w) = \emptyset$ then \Cref{thm linearization} implies that $(0,T) = \ddeg(w)$.
\end{proof}

\begin{remark}
	Note that the interior ball condition only needs to hold in a neighborhood close to $w$.
\end{remark}

If the domain in addition to satisfying the interior ball condition also satisfies the so-called NTA condition (see \cite{JK}), we can apply the Carleson estimate developed in \cite{AKN} (or even \cite{A}) to conclude that the non-tangential limsup in \Cref{thm dichotomy} can be replaced with the regular limsup from inside $\Omega$. This rules out odd behavior in tangential directions. Furthermore we obtain that $\ddegg(w)$ (defined below) is an open interval.

\begin{theorem} \label{thm complete dichotomy}
	Let $\Omega \subset \R^N$ be an NTA-domain with constants $M,r_0$, satisfying the interior ball condition with radius $r_0 > 0$. 
	Let $u$ be a non-negative solution to \cref{Hu} in $\Omega_T$ vanishing continuously on a neighborhood of $\set{w} \times (0,T)$ in $\partial \Omega \times (0,T)$.
	Let $w \in \partial \Omega$, and define the sets
	\begin{align*}
		\ddegg(w) &:= \biggset{t \in (0,T): \limsup_{\Omega \ni x \to w}\bigg [ \frac{u(x,t)}{|x-w|^{\frac{p}{p-2}}} \bigg ] < \infty} \\
		\dpar(w) &:= \biggset{t \in (0,T): \limsup_{\Omega \ni x \to w} \frac{d(x,\partial \Omega)}{u(x,t)} < \infty},
	\end{align*}
	then $\ddegg(w)$ and $\dpar(w)$ are disjoint intervals. If $\ddegg(w)$ and $\dpar(w)$ are both nonempty, then there exists a time $\hat t \in (0,T)$ such that
	\begin{equation} \notag %\label{eq 0Tsplit}
		(0,T) = \ddegg(w) \cup \set{\hat t} \cup (\dpar(w))^\circ, \quad (0,\hat t) = \ddegg(w),
	\end{equation}
	and the union is disjoint.
	Otherwise if $\ddegg(w) = \emptyset$ or $\dpar(w) = \emptyset$ then
	\begin{equation} \nonumber
		(0,T) = \dpar(w) \cup \ddegg(w).
	\end{equation}
\end{theorem}
\begin{proof}
	Consider now a point $w \in \partial \Omega$, then from \Cref{thm dichotomy} it follows that unless $t = \hat t$ we have $t \in \ddeg(w)$ or $t \in \dpar(w)$. Let us prove that if $t \in \ddeg(w)$ $(0,t) \subset \ddegg(w)$.
	
	Let $r < r_0$ and assume that
	\begin{equation} \label{eq large nontan}
		\mathcal{N}^+\bigg [ \frac{u(x,t_0)}{|x-w|^{\frac{p}{p-2}}} \bigg ] (w) < \infty
	\end{equation}
	for some $t_0 \in (0,T)$. As mentioned in \cref{rem:Non-tangential} we know that $a_\varrho(w) \in \Gamma(w)$ for $\varrho < r$. Thus from \Cref{eq large nontan} there exists a constant $\hat C$ such that
	\begin{equation} \nonumber
		u(a_\varrho(w),t_0) \leq \hat C \varrho^{\frac{p}{p-2}} =: \Lambda_\varrho.
	\end{equation}
	With this at hand let us calculate $\tau_\varrho$ from \Cref{thm Carleson} as follows
	\begin{equation} \nonumber
		\tau_\varrho = \frac{C_4}{4} \left[C_5 \Lambda_\varrho \right]^{2-p} \varrho^p = \frac{C_4}{4} \left[C_5 \hat C \right]^{2-p}.
	\end{equation}
	Let now $\epsilon > 0$ be an arbitrary parameter, then take $\delta_1^{p-1} = \delta_2^{p-1} = \delta_3^{p-1} = \epsilon/(4\tau_\varrho)$, in \Cref{thm Carleson} thus if $t_0 - \epsilon \in (0,T)$ and $u$ vanishes at the boundary piece $(\partial \Omega \cap B_r(w)) \times (t_0-\epsilon,t_0)$ then from \Cref{thm Carleson} we have
	\begin{equation} \nonumber
		\sup_{(\Omega \cap B_\varrho(w)) \times (t_0-\epsilon/2,t_0-\epsilon/4)} u \leq c_9(\epsilon) \Lambda_\varrho = c_{10} \varrho^{\frac{p}{p-2}}.
	\end{equation}
	Hence we can conclude that for $t \in (t_0-\epsilon/2,t_0-\epsilon/4)$ we have
	\begin{equation} \nonumber
		\limsup_{\Omega \ni x \to w}\bigg [ \frac{u(x,t)}{|x-w|^{\frac{p}{p-2}}} \bigg ] < \infty.
	\end{equation}
	Since $\epsilon$ and $t$ was arbitrary we obtain $(\ddeg(w))^{\circ} \subset \ddegg(w)$.
	
	Finally we prove that $\ddegg(w)$ is an open interval, to do this let us assume that $\ddegg(w)$ is closed, i.e. we know that there exists a $\hat t$ such that $\ddegg(w) = (0,\hat t]$. This implies that
	\begin{equation} \nonumber
		C_{m} := \sup_{[\hat t/2,\hat t]} \limsup_{\Omega \ni x \to w}\bigg [ \frac{u(x,t)}{|x-w|^{\frac{p}{p-2}}} \bigg ] < \infty
	\end{equation}
	for each $\epsilon$ there is a $\delta$ such that
	\begin{equation} \nonumber
		C_m \leq \sup_{[\hat t/2,\hat t]} \sup_{B_\delta(w) \cap \Omega}\bigg [ \frac{u(x,t)}{|x-w|^{\frac{p}{p-2}}} \bigg ] \leq C_m+\epsilon
	\end{equation}
	i.e. there is a new constant $C_m$ such that
	\begin{equation} \nonumber
		u(x,t) \leq C_m |x-w|^{\frac{p}{p-2}}
	\end{equation}
	in $(B_\delta(w) \cap \Omega) \times [\hat t/2, \hat t]$. Thus it is easy to see that we can apply \cref{thm memory full} to obtain that there is a time $t > \hat t$ such that $t \in \ddegg(w)$ and therefore we arrive at a contradiction. 
\end{proof}

\subsection{Classifying the support of the boundary type Riesz measure}
There is a true equivalence of support of the measure $\mu$ and the regions "non-degenerate" and the "degenerate". In effect if a point becomes non-degenerate it will permeate to the whole domain in finite time and provide support for the boundary type Riesz measure so a non-degenerate point is a true critical point for the behavior of the equation. However for simplicity we will in the following theorem consider when we have degeneracy on a space-time set on the boundary and prove that the measure vanishes.

\begin{theorem} \label{thm:Riesz_no_support}
	Let $\Omega \subset \R^N$ be a bounded domain, and let $u$ be a non-negative solution to \cref{Hu} in $\Omega_T$. Let $E := B_r(w) \times (t_0-\epsilon,t_0)$, $(t_0-\epsilon,t_0) \subset (0,T)$, assume that $u$ vanishes continuously on $E \cap \partial_p \Omega_T$, and 
	\begin{equation} \nonumber
		\limsup_{\Omega \ni x \to w}\bigg [ \frac{u(x,t)}{|x-w|^{\frac{p}{p-2}}} \bigg ] < \infty, \quad \forall (w,t) \in E \cap \partial_p \Omega_T.
	\end{equation}
	Then the measure defined in \Cref{eq:Riesz} vanishes on $E$.
\end{theorem}
\begin{remark} \label{rem:cont}
	The above theorem implies that for a solution that has a section of degenerate points can actually be extended across the boundary as a solution. This is a fairly remarkable result and provides an example of solutions with a free boundary that is stationary for a positive time interval. I.e. consider
	\begin{align*}
		u_1(x,t)&=C(p)(T-t)^{-\frac{1}{p-2}%
		}\max\{x_{n},0\}^{\frac p{p-2}}
	\end{align*}
	which is a solution across $x_n = 0$ according to \cref{thm:Riesz_no_support}. In fact it is an example of an ancient solution with a stationary free boundary. 
	
	Another example would be if we had initial datum satisfying
	\begin{align*}
		u_0(x) \leq d(x,\partial \Omega)^{\delta}, \qquad \delta \geq \frac{p}{p-2},
	\end{align*}
	in an NTA-domain $\Omega$ and consider a solution in $\Omega \cap B_r(w) \times (0,T)$, $w \in \partial \Omega$ such that $u \leq M$ and $u$ vanishing on $B_r(w) \cap \partial \Omega \times (0,T)$, then if $T$ is small enough depending only on $M,p,\delta$, this solution can be extended across the boundary as a solution (see \cref{thm memory full,thm:Riesz_no_support}). This solution is a more advanced example of a solution with a stationary free boundary.
\end{remark}
\begin{theorem} \label{thm:Riesz_support}
	Let $\Omega \subset \R^N$ be a bounded NTA-domain with constants $M,r_0$, and let $u$ be a non-negative solution to \cref{Hu} in $\Omega_T$. Let $E := B_r(w) \times (t_0-\epsilon,t_0)$, $(t_0-\epsilon,t_0) \subset (0,T)$, assume that $u$ vanishes continuously on $E \cap \partial_p \Omega_T$, and 
	\begin{align} \label{eq:Riesz_support_upper_bound}
		\sup_{E \cap \Omega_T} \frac{d(x,\partial \Omega)}{u(x,t)} = \Lambda^{-1} < \infty,
	\end{align}
	then $\mu$ is positive on $E \cap \partial_p \Omega_T$.
\end{theorem}

% \subsection{Techniques} % (fold)
% This paper uses very basic techniques, we will rely heavily on the comparison principle, scaling properties and several explicit sub-/super-solutions. This is because many of the particular explicit solutions I have used in this paper exists in similar versions w.r.t. different equations, for example the doubly nonlinear equation and the PME. However for those equations greater care has to be taken as we cannot anymore add constants to our solution and we must in \Cref{lower bound1} find another explicit solution with the properties we are looking for, I do believe that the one we constructed in \cite{AKN} could be modified for these equations.
%
% \subsection{$p$-Stability}
% It should be noted that no care has been taken to make the estimates $p$-stable when $p \to 2$. Since I am looking for a phenomenon that only occurs for $p > 2$, I have decided for the sake of simplicity to make most estimates $p$-unstable.

\subsection{Consequences for Harnack chains} \label{cons:harnack}
We begin by assuming that we have a solution to the $p$-parabolic equation in $\Omega_T$, where $\Omega \subset \R^n$ is an $(M,r_0)$-NTA-domain. We assume that for a given instant $t_0 \in (0,T)$ and a given point on the boundary $w \in \partial \Omega$ the following estimate holds from below
\begin{align} \label{cons:lower}
	r \leq C u(a_r(w),t_0) \quad \text{for all $r \in (0,r_0)$.}
\end{align}
Let $y \in \Omega$ be any point and assume that $d(y,\partial \Omega) \approx d(a_{r}(w),y) \leq L r$ holds for some $r \in (0,r_0)$ then the following holds:
\begin{align*}
	u(a_r(w),t_0) \leq c u(y,t_0+C r^2), \quad \text{for some $C > 0$ depending on $L$.}
\end{align*}
Since $\Omega$ is an NTA-domain this follows from \cite[Theorem 3.5]{AKN} together with the lower bound \cref{cons:lower}. The immediate consequence of this is that most Harnack-based estimates from below reduces to a non-intrinsic version and scales exactly as in the linear case (heat equation). This is yet another reason for denoting the estimate \cref{cons:lower} a non-degeneracy estimate.

\section{Immediate linearization: Proof of Theorem \ref{thm linearization}}
We begin this section with a barrier type argument together with a rescaling and iteration method to obtain a ``sharp'' lower bound of a useful comparison function. A bit more complicated proof, which is $p$-stable can be found in \cite{AKN}. In the proof below we are employing the Barenblatt solution, which is given by,
\newcommand{\U}{\mathcal U}
\begin{align*}
	\U(x,t) = t^{-k} \left ( C_0 - q \left ( \frac{|x|}{t^{k/n}}\right )^{\frac{p}{p-1}} \right )^{\frac{p-1}{p-2}}_+
\end{align*}
where
\begin{align*}
	k = \left ( p-2 + \frac{p}{n} \right )^{-1}, \qquad q = \frac{p-2}{p} \left ( \frac{k}{n}\right )^{\frac{1}{p-1}}
\end{align*}
and $C_0$ is a constant depending only on $p,n$. In the following proof we will be using some properties of the Barenblatt function, the first is that the level sets are strictly increasing balls with time, the second is that the maximum of the function is for each time slice at the origin, the third is that the radial derivative is non-zero as long as the function is positive and we are not at the origin.

\begin{lemma}\label{lower bound1}
	Consider a solution to
	% \begin{equation} \label{initial bastard}
	% 	\begin{cases}
	% 		u_t - \Delta_p u = 0, &\text{ in $B_2 \times (0,\infty)$}\\
	% 		u = 0, &\text{ on $\partial B_2 \times (0,\infty)$}\\
	% 		u \geq \chi_{B_1}, &\text{ for $t=0$}\\
	% 	\end{cases}
	% \end{equation}
	\begin{align} \label{initial bastard}
		\left \{
		\begin{array}{rcll}
			u_t - \Delta_p u &=& 0, & \text{ in $B_2 \times (0,\infty)$} \\
			u &=& 0, &\text{ on $\partial B_2 \times (0,\infty)$}\\
			u &\geq& \chi_{B_1}, &\text{ for $t=0$}\\
		\end{array}
		\right .
	\end{align}
	then there exists constants $c_1,c_2,c_3,c_4 > 0$ all depending only on $p$ and $n$ such that
	\begin{equation} \label{inflowerbound}
		\inf_{x \in B_1} u(x,t) \geq c_1\bigg (c_2 t+1\bigg )^{\frac{1}{2-p}}, \quad t \geq 0,
	\end{equation}
	and 
	\begin{equation} \label{distancelowerbound_simple}
		u(x,t) \geq c_3 (2-|x|) \bigg (c_2 t+1\bigg )^{\frac{1}{2-p}}, \quad t \geq c_4.
	\end{equation}
\end{lemma}
\begin{proof}
	Let us consider the Barenblatt function $\U$, then consider $t_1$ such that 
	\begin{equation} \nonumber
		\supp \U(\cdot,t_1) = B_1,
	\end{equation}
	denote $\U(0,t_1) = \kappa_1$, and consider the rescaled Barenblatt
	\begin{equation} \nonumber
		\hat \U(x,t) = \frac{\U(x,\kappa_1^{2-p} t+t_1)}{\kappa_1}.
	\end{equation}
	Now let $t_2$ be such that $\supp \hat \U(\cdot,t_2) = B_2$, then denote
	\begin{equation} \nonumber
		\sigma = \inf_{x \in B_1} \hat \U(x,t_2) > 0.
	\end{equation}
	Note that $\sigma = \sigma(p,n) \in (0,1)$ is a constant.
	By the comparison principle we obtain that in $B_2 \times (0,t_2)$ we have that $u \geq \hat \U$ and thus
	\begin{equation} \label{lower1}
		\inf_{x \in B_1} u(x,t_2) \geq \sigma.
	\end{equation}

	We will now use an iterative argument. Assume that we have a function $u$ in $B_2 \times (t_0,\infty)$ such that
	\begin{equation} \label{iter1}
		\inf_{x \in B_1} u(x,t_0) \geq m
	\end{equation}
	then the function 
	\begin{equation} \label{iterrescale}
		v(x,t) = \frac{u(x,m^{2-p} t+t_0)}{m}
	\end{equation}
	satisfies \Cref{initial bastard} and thus from \Cref{lower1} 
	\begin{equation} \nonumber
		\inf_{x \in B_1} v(x,t_2) \geq \sigma.
	\end{equation}
	Rescaling back we obtain
	\begin{equation} \label{iter2}
		\inf_{x \in B_1} u(x,t_0+t_2 m^{2-p}) \geq m \sigma.
	\end{equation}
	
	If we start from $u$ in \Cref{initial bastard} with $t=0$ and iterate \Cref{iter1,iter2,iterrescale}, each starting from the following times 
	\begin{equation} \label{itersum1}
		\tau_k = \sum_{i=1}^k t_2 m_{i-1}^{2-p} = t_2 \sum_{i=0}^{k-1} (\sigma^{2-p})^i = t_2 \frac{(\sigma^{2-p})^{k-2}-1}{\sigma^{2-p}-1},
	\end{equation}
	where $m_i := \sigma^i$, we get
	\begin{equation} \label{itersum2}
		\inf_{x \in B_1} u(x,t_0+\tau_k) \geq m_{k}.
	\end{equation}
	Let us be given a $t \geq 0$, and let $k$ be the integer such that
	\begin{equation} \label{tauk}
		\tau_{k-1} \leq t \leq \tau_k
	\end{equation}
	then by \Cref{itersum1,tauk}
	\begin{equation} \label{tktranslation1}
		\sigma^2 \bigg (t \frac{(\sigma^{2-p}-1)}{t_2}+1\bigg )^{\frac{1}{2-p}} \in (\sigma^{k},\sigma^{k-1})
	\end{equation}
	thus collecting \Cref{tauk,itersum1,itersum2,tktranslation1} we get \Cref{inflowerbound}.
	
	Let us now prove \Cref{distancelowerbound_simple}. This we do as follows.
	Consider again the Barenblatt function, with $t_1$ as before, but this time, let us consider the following rescaled Barenblatt
	\begin{equation} \nonumber
		\tilde \U(x,t) = \frac{\U(x,(\kappa_1/2)^{2-p}t+t_1)-\kappa_1/2}{\kappa_1/2}.
	\end{equation}
	Then find $\tilde t_2$ where $\set{\tilde \U > 0} = B_2$, and note that $\tilde t_2 > t_2$ and thus there is a constant $\tilde c > 0$ such that
	\begin{equation} \label{tildet2lower}
		\tilde \U(x,\tilde t_2) \geq \tilde c (2-|x|).
	\end{equation}
	Now, by the construction of $\tilde \U$ and the parabolic comparison principle we get that in $\overline{B_2} \times [0,\tilde t_2]$, $\tilde U \leq u$ and thus \Cref{tildet2lower} holds also for $u$ at $\tilde t_2$.
	
	Going back to \Cref{itersum1}, let $k_0 > 0$ be the number such that
	\begin{equation} \nonumber
		\tau_{k_0} \leq \tilde t_2 \leq \tau_{k_0+1}.
	\end{equation}
	Consider $v(x):B_2 \setminus B_1 \to \R_+$ as the solution to $\Delta_p v = 0$ and $v = c_0$ at $\partial B_1$ and $0$ at $\partial B_2$, where $c_0$ is to be fixed. First take $c_0$ to be the largest number so that
	\begin{equation} \nonumber
		v \leq \tilde c (2-|x|) \quad \text{and}\quad c_0 \leq \sigma^{k_0+1}.
	\end{equation}
	This implies that $v(x) \leq u(x,t)$ on $\partial_p [ (B_2 \setminus B_1) \times (\tilde t_2, \tau_{k_0+1})]$ and thus by the parabolic comparison principle we have $v \leq u$ in $(B_2 \setminus B_1) \times (\tilde t_2,\tau_{k_0+1}]$. Moreover there is a new constant $\hat c > 0$ such that
	\begin{equation} \nonumber
		v \geq \hat c (2-|x|).
	\end{equation}
	Now we can apply the same argument iteratively for $(\tau_{k_0+j},\tau_{k_0+j+1})$, $j=1,\ldots$ and obtain
	\begin{equation} \nonumber
		u(x,t) \geq c_2 (2-|x|) \sigma^j
	\end{equation}
	for $t \in [\tau_{k_0+j},\tau_{k_0+j+1}]$, and thus we get as in \Cref{tauk,tktranslation1} that \Cref{distancelowerbound_simple} holds, and thus we have proved \Cref{lower bound1}.
\end{proof}

% \begin{theorem} (Immediate linearization)\label{thm linearization repeat}
% 	Let $\Omega \subset \R^N$ be a domain satisfying the interior ball condition with radius $r_0 > 0$. Let $u$ be a non-negative solution to \eqref{ppar1} in $\Omega_T$ and assume that $w \in \partial \Omega$ and $t_0 \in (0,T)$
% 	\begin{equation} \label{eq linearization req}
% 		\mathcal{N}^+\left [ \frac{u(x,t_0)}{|x-w|^{\frac{p}{p-2}}} \right ] (w) = \infty
% 		%\limsup_{\Gamma(w) \ni x \to w} \frac{u(x,t)}{|x-w|^{\frac{p}{p-2}}} = \infty\,.
% 	\end{equation}
% 	%for some function $\eta:\R_+ \to \R_+$ where $\eta(0^+) = \infty$.
% 	Then for any $t_+ \in (t_0,T)$ we have
% 	\begin{equation} \label{eq linearized}
% 		\limsup_{\Omega \ni x \to w} \frac{d(x,\partial \Omega)}{u(x,t_+)} < \infty\,.
% 	\end{equation}
% \end{theorem}

\begin{proof}[Proof of \cref{thm linearization}]
	Due to translation invariance we can assume that $w = 0$. Let $\epsilon > 0$ be a given number such that $t-\epsilon > 0$ and $t+\epsilon < T$.
	The condition \Cref{eq linearization req} implies that there is a sequence of points $\set{x_j}_{j=1}^{\infty}$, $x_j \in \Gamma(w)$ and $x_j \to w$ such that for a strictly decreasing function $\eta$ we get
	\begin{equation} \label{eq eta definition}
		u(x_j,t_0) \geq \eta(|x_j|) |x_j|^{\frac{p}{p-2}},
	\end{equation}
	where $\eta(0^+) = \infty$.
		%
	% First note that we can redefine $\eta$ such that it is a strictly decreasing continuous function.
	% Note that $r \leq 8$. For each $j \in \{0,\ldots\}$ define
	% \begin{equation} \nonumber
	% 	\tilde \eta(2^{-j+3}) = \sup_{t \in [2^{-j+3},8]} \eta(t)
	% \end{equation}
	% Now consider a $t \in (0,8)$ and let $j$ be such that
	% \begin{equation} \nonumber
	% 	a:= 2^{-j-1+3} \leq t \leq 2^{-j+3} =: b
	% \end{equation}
	% and define
	% \begin{equation} \nonumber
	% 	\tilde \eta(t) = \tilde \eta(a)+\frac{\tilde \eta(b)-\tilde \eta(a)}{b-a} (t-a)\,.
	% \end{equation}
	% We now see that we can assume that $\eta$ is a non-increasing continuous function with $\eta(0^+) = \infty$.

	Start by considering $g(r) = r^{\frac{p}{p-2}} \eta(r)$, and define the ``time-lag'' function $\theta$ as
	\begin{equation} \nonumber
		\theta(r):= g(r)^{2-p} r^p = r^{(p-2)\left (\frac{p}{p-2}-\frac{p}{p-2} \right )}\eta(r)^{2-p} = \eta(r)^{2-p} 
	\end{equation}
	which implies that $\theta$ is a strictly increasing function such that $\theta(0^+) = 0$.
	We construct the sets for $r \leq r_0$
	\begin{align*}
		N^r_1 &:= \bigset{x \in \Omega: d(x,\partial \Omega \cap B(0,r)) = d(x,\partial \Omega) = r}, \\
		N^r_2 &:= \bigset{x \in \Omega: d(x,N^r_1) \leq r/2}.
	\end{align*}
	Note that since $\Omega$ is an NTA-domain we see that $N^r_1$ can be covered by $k(M,r_0)$ balls of size $r/4$. 
	Now, consider the unique $r_\epsilon$ such that
	\begin{equation*}% \label{repsilon}
		\theta(r_\epsilon) = \frac{\epsilon}{4 \max \set{\bar c_1, c_4(\bar c_2^{k})^{p-2}, \bar c_3(k)}},
	\end{equation*}
	where $c_4 > 1$ is from \Cref{lower bound1} and $\bar c_i$, $i \in \set{1,2,3}$ is from \Cref{lem AKN HChain}. Let $J \in \mathbb{N}$ be the smallest integer such that $|x_J| \leq \min\set{r_\epsilon,r_0}$. Denote $r = d(x_J,\partial \Omega)$ and note that since $x_J \in \Gamma(w)$, $r \approx |x_J|$. 
	From \Cref{eq eta definition} and the choice of $r_\epsilon$ we can apply the forward Harnack chain (\Cref{lem AKN HChain}) to obtain that in $N^r_2$ we have for a time 
	\begin{equation} \nonumber
		t_1 = t_0 + \bar c_3(k) u(x_J,t_0)^{2-p} \left (\frac{r}{4} \right )^p,
	\end{equation}
	where $t_0 < t_1 < t_0+\epsilon/4$ and
	\begin{equation} \label{eq u larger than g r}
		u(x,t_1) \geq \bar c_2^{-k} u(x_J,t_0) \geq \bar c_2^{-k} g(r).
	\end{equation}
	Now consider any point $y$ in $N^r_1$, and let $v$ denote a function satisfying \cref{initial bastard} but with $v(\cdot,0) = \chi_{B_1}$. Let us translate and scale $v$ as follows 
	\begin{equation} \nonumber
		\bar v(x,t) = (\bar c_2^{-k} g(r)) v\left (\frac{x-y}{r/2}, (t-t_1) \frac{(\bar c_2^{-k})^{p-2} 2^p}{\theta(r)} \right ).
	\end{equation}
	The function $\bar v$ now satisfies
	\begin{align*}
		\left \{
		\begin{array}{rcll}
			\bar v_t - \Delta_p \bar v &=& 0, & \text{ in $B_r \times (t_1,\infty)$} \\
			\bar v &=& 0, &\text{ on $\partial B_r \times (t_1,\infty)$}\\
			\bar v &=& \bar c_2^{-k} g(r)\chi_{B_{r/2}}, &\text{ for $t=t_1$}.
		\end{array}
		\right .
	\end{align*}
	We can now use the parabolic comparison principle together with \Cref{eq u larger than g r} to conclude $u \geq \bar v$. Applying \Cref{lower bound1} we get for
	\begin{equation} \nonumber
		t_2 := t_1 + c_4 \frac{\theta(r)}{(\bar c_2^{-k})^{p-2} 2^p} < t < T
	\end{equation}
	that
	\begin{equation} \nonumber %\label{distancelowerbound}
		u(x,t) \geq c_3 \frac{2(r-|x-y|)}{r} \bigg (c_2 (t-t_1) \frac{(\bar c_2^{-k})^{p-2} 2^p}{\theta(r)}+1\bigg )^{\frac{1}{2-p}}.
	\end{equation}
	Furthermore our choice of $r_\epsilon$ gives $t_0 < t_2 < t_0 + \epsilon$.
	In particular since $y$ was an arbitrary point in $N^r_1$ and $\epsilon > 0$ was arbitrary, we can conclude that \Cref{eq linearized} holds for $t > t_0+\epsilon$.
\end{proof}

\section{Memory-effect for degenerate initial data: Proof of Theorem \ref{thm memory full}} % (fold)
\label{sec:memory}
This next result is a theorem about a certain memory effect of the $p$-parabolic equation, essentially it states that if the boundary behavior at a fixed point has a certain decay-rate property that is higher than $r^{\frac{p}{p-2}}$ then the equation will remember this decay for some time forward dictated by the size of the solution. Another way to look at this is that a solution with this decay-rate does not regularize immediately.

\begin{theorem}\label{thm memory}
	Consider a bounded domain $\Omega \subset \R^n$ and assume that $w \in \partial \Omega$. Consider the domain $E_T := E \times (0,T) := \Omega \cap B(w,2) \times (0,T)$ and consider a solution to the following Cauchy-Dirichlet problem
	\begin{equation} \label{upperproblem}
		\left \{
		\begin{array}{rcll}
			u_t - \Delta_p u &=& 0, &\text{ in $E_T$ }\\
			u &=& 0, &\text{ on $\partial \Omega \cap \partial_p E_T$}\\
			u(x,t) &\leq& M, &\text{ on $\partial_p E_T$ } \\
			u(x,0) &\leq& M |x-w|^\delta,&\text{ for $x \in E$}.
		\end{array}
		\right .
	\end{equation}
	Then there exists a time
	\begin{equation} \label{hatT}
		\hat T := \left [\frac{C(p,\delta)}{M} \right ]^{p-2}
	\end{equation}
	and a constant $c_0(p) > 0$	such that for $t \in (0,\min\set{\hat T/2,T})$ the following upper bound holds
	\begin{equation} \nonumber
		u(x,t) \leq c_0 M |x-w|^{\delta}, \quad x \in E.
	\end{equation}
\end{theorem}
\begin{proof}
	In order for us to be able to work in higher dimensions than $n=1$ we need to construct a radial version of \Cref{examp1}. For this let us consider a solution of the type $v(x,t):=g(|x|)f(t)$, plugging this into equation \cref{Hu} gives us
	\begin{equation} \nonumber % \label{split1}
		f_t g - f^{p-1} \Delta_p g = 0.
	\end{equation}
	Let us first solve
	\begin{equation} \nonumber %\label{tsplit}
		f_t(t) = f^{p-1}(t),
	\end{equation}
	which has as a solution $c_p (T-t)^{-\frac{1}{p-2}}$ for any value of $T$,  let us use $T = 1$.
	Next let us solve
	\begin{equation} \nonumber %\label{unradialsplit}
		\Delta_p g \leq g
	\end{equation}
	which in radial form $|x|=r$ looks like
	\begin{equation} \label{radialsplit}
		|g'|^{p-2} \bigg [(p-1) g'' +\frac{n-1}{r}g' \bigg ] \leq g.
	\end{equation}
	Now let $\delta \geq \frac{p}{p-2}$ then for $g = c_0 r^\delta$
	\begin{equation} \nonumber
		|g'|^{p-2}\bigg [(p-1) g'' +\frac{n-1}{r}g' \bigg ] = c_0^{p-1} \delta^{p-1} \bigg [(p-1) (\delta-1)  +(n-1) \bigg ]r^{\delta(p-1)-p}
	\end{equation}
	So we only need to choose $c_0$ to be%
	\begin{equation} \label{cochoice}
		c_0(p,\delta,n) := \bigg[ \delta^{p-1} \big [(p-1) (\delta-1)  +(n-1) \big ]  \bigg ]^{\frac{1}{2-p}} 2^{\frac{p}{p-2}-\delta},
	\end{equation}
	in order for $g$ to satisfy \Cref{radialsplit} for $r \leq 2$, since $\delta(p-1)-p \geq \delta$. In fact for our choice of $c_0$ we see that \Cref{radialsplit} is solved with an equality if $\delta = \frac{p}{p-2}$, just as in the one dimensional case \Cref{examp1}. Specifically
	\begin{equation} \nonumber %\label{finalfunction}
		v(x,t) = c_p c_0 (1-t)^{-\frac{1}{p-2}} |x|^{\delta}
	\end{equation}
	is a supersolution to \Cref{ppar1} in $(B_2 \setminus \set{0}) \times (-\infty,1)$.
	
	Let now $u$ be as in \cref{upperproblem} and assume for simplicity that $w = 0$. Let $v$ be the supersolution established above and consider the rescaled function $\bar v$ as
	\begin{align*}
		\bar v(x,t) = \frac{M}{c_0 c_p} v\left ( x, \left ( \frac{c_0 c_p}{M} \right )^{2-p} t \right )
	\end{align*}
	then if we define $C(p,\delta) = c_0 c_p$ and use $\hat T$ as in \cref{hatT}, we get that $\bar v$ satisfies
	\begin{equation} \nonumber %\label{upperproblembarrier}
		\left \{
		\begin{array}{rcll}
			\bar v_t - \Delta_p \bar v &=& 0, &\text{ in $E_{\hat T}$ }\\
			\bar v &\geq& 0, &\text{ on $\partial \Omega \cap \partial_p E_{\hat T}$}\\
			\bar v(x,t) &\geq& M, &\text{ on $\partial_p E_{\hat T}$ } \\
			\bar v(x,0) &=& M |x|^\delta,&\text{ for $x \in E$}.
		\end{array}
		\right .
	\end{equation}
	Thus by the parabolic comparison principle the conclusion of the theorem follows.
\end{proof}

The proof of \cref{thm memory full} now follows from \cref{thm memory} by a simple scaling argument.

\section{The boundary type Riesz measure}
We begin this section with a simplified version of the upper bound of the measure (see \cref{eq:Riesz}) that we developed in \cite[Theorem 5.2]{AKN}, we have included the proof for ease of the reader.
\begin{lemma}
	\label{MuUpperBound} Let $\Omega \subset \R^n$ be a bounded domain. Let $0<r$ and let $u$ be a non-negative solution to \cref{Hu} in $\Omega_T$. Fix a point $x_0 \in \partial \Omega$, $\delta  \in (0,1)$ and a $\Lambda > 0$ such that
	\begin{equation*}
		Q = B_r(x_0) \times (t_0 - \delta \Lambda^{2-p} r^p , t_0),
	\end{equation*}
	\begin{equation} \nonumber
		\Lambda \geq \sup_{2Q \cap \Omega_T} u,
	\end{equation}
	and $t_0 - 2\delta \Lambda^{2-p} r^p > 0$. Now if $u$ vanishes on $\partial_p \Omega_T \cap 2Q$ then the following upper bound holds ($\mu$ is as in \cref{eq:Riesz})
	\begin{align*}
		\frac{\mu(Q)}{r^n} \leq C \Lambda.
	\end{align*}
\end{lemma}
\begin{proof}
	As in the construction of the measure $\mu$ in \Cref{eq:Riesz}, we see that extending $u$ to the entire cylinder $Q$ as zero, we obtain a weak subsolution \cref{Hu} in $Q$.	
	Take a cut-off function $\phi \in C^\infty(2Q)$ vanishing on $\partial_p 2Q$ such that $0 \leq \phi \leq 1$, $\phi$ is one on $Q$, and $|\grad \phi| < C/r$ and $(\phi_t)_+ < C \frac{\Lambda^{p-2}}{\delta r^p}$. Then by~\cref{eq:Riesz}, the definition of~$\phi$ and H{\"o}lder's inequality we get
	\begin{align*}
		\int_{2Q} \phi^p \de \mu &\leq \int_{2Q} |\grad u|^{p-1} |\grad \phi| \phi^{p-1} \dx \dt + \int_{2Q} u (\phi_t)_+ \phi^{p-1}\dx \dt \\
		&\leq \frac{4}{r} \int_{2Q} |\grad u|^{p-1} \phi^{p-1} \dx \dt + \int_{2Q} u (\phi_t)_+ \phi^{p-1} \dx\dt \\
		&\leq \frac{4}{r} |2Q|^{1/p} \left ( \int_{2Q} |\grad u|^p \phi^p \dx \dt \right )^{\frac{p-1}{p}} + \int_{2Q} u (\phi_t)_+ \phi^{p-1} \dx \dt .
	\end{align*}
	Now using the standard Caccioppoli estimate
	\begin{align*}
		\mu(Q) &\leq C \frac{|2Q|^{1/p}}{r} \left ( \int_{2Q} u^p |\grad \phi|^p + u^2 (\phi_t)_+ \phi^{p-1} \dx \dt \right )^{\frac{p-1}{p}} + \int_{2Q} u (\phi_t)_+ \phi^{p-1} \dx \dt , \\
		&\leq C \frac{|2Q|^{1/p}}{r} \left (|2Q| \left (\frac{\Lambda^p}{r^{p}} + \frac{\Lambda^p}{\delta r^p } \right ) \right )^{\frac{p-1}{p}} + C |2Q| \frac{\Lambda^{p-1}}{\delta r^p} \\
		&\leq C \frac{|2Q|}{\delta r^p} \Lambda^{p-1} = C \frac{r^n \delta \Lambda^{2-p} r^p}{\delta r^p} \Lambda^{p-1} \leq C r^n \Lambda .
	\end{align*}
\end{proof}

We are now ready to tackle the proof of \cref{thm:Riesz_no_support}, which just utilizes the above estimate to get the radius dependency explicit such that when considering a covering will just imply that the measure has no support and its restriction to the set of degeneracy is simply zero.

\begin{proof}[Proof of \cref{thm:Riesz_no_support}]
	Let $\hat E = \hat B \times \hat I$ be a space time cylinder such that $\hat E \Subset E$. The assumptions on $u$ gives that there exists a constant $C$ such that
	\begin{equation} \label{eq meas zero 1}
		u(x,t) \leq C |x-w|^{\frac{p}{p-2}}, \quad (x,t) \in \hat E, w \in \hat B \cap \partial \Omega.
	\end{equation}
	Consider a cube $Q_{\varrho,\sigma}(y,s) = (\partial \Omega \cap B_{\varrho}(y)) \times (s-\sigma,s)$ such that $2Q_{\varrho,\sigma}(y,s) \subset \hat E \cap \partial_p \Omega_T$. Note that \Cref{eq meas zero 1} gives
	\begin{equation} \nonumber
		\sup_{2Q_{\varrho,\sigma}(y,s)} u \leq C_0 \varrho^{\frac{p}{p-2}}.
	\end{equation}
	Thus setting $\Lambda = C_0 \varrho^{\frac{p}{p-2}}$ and setting $\delta = C_0^{p-2} \sigma$ we get from \Cref{MuUpperBound} that
	\begin{equation} \nonumber 
		\mu(Q_{\varrho,\sigma}(y,s)) \leq C(\sigma) \varrho^{n+\frac{p}{p-2}}.
	\end{equation} 
	Now, since the height is fixed as $\sigma$ irrespective of the radius $\varrho > 0$ the decay-rate of the measure is greater than $\varrho^n$ which simply implies that it is a zero measure inside $\hat E$.
\end{proof}

For the proof of \cref{thm:Riesz_support} we will be needing the following two lemmas from \cite{AKN}.

\begin{lemma}
	\label{MuLowerSimple} Let $\Omega \subset \R^n$ be an NTA-domain with constants $M$ and $r_0=2$. There exists constants $C,T_0$, both depending on $p,n,M$, such that if $v$ is a continuous solution to the problem
	\begin{equation*}
		\label{eq:vDirichlet}
		\begin{cases}
			v_t - \Delta_p v = 0 &\text{ in } (\Omega \cap B_2(0)) \times (0,T_0) \\
			v = 0 &\text{ on } \partial (\Omega \cap B_2(0)) \times [0,T_0) \\
			v = \chi_{B_{1/(4M)}(a_1(0))} & \text{ on } (\Omega \cap B_2(0)) \times \{0 \},
		\end{cases}
	\end{equation*}
	then
	\begin{equation*}
		\mu_v\big(B_2(0) \times (0,T_0)\big) \geq 1/C.
	\end{equation*}
	Furthermore, constants $C, T_0$, are stable as $p \to 2^+$
\end{lemma}

\begin{lemma}
	\label{meas order} Let $\Omega \subset \R^n$ be a domain. Let $u$ and $v$ be weak solutions in $(\Omega \cap B_r(0)) \times (0,T)$ such that $u \geq v \geq 0$ and both vanish continuously on the lateral boundary $(\partial \Omega \cap B_r(0)) \times (0,T)$. Then
	\begin{equation*}
		\mu_v \leq \mu_u \qquad \mbox{in }  B_r(0) \times (0,T)
	\end{equation*}
	in the sense of measures.
\end{lemma}

\begin{proof}[Proof of \cref{thm:Riesz_support}]
	From \cref{eq:Riesz_support_upper_bound} we see that
	\begin{align*}
		u(x,t) \geq \Lambda d(x,\partial \Omega).
	\end{align*}
	Let us now consider $(t_1,t_2) \Subset (t_0-\epsilon,t_0)$ and a point $a_{\varrho}(y)$ for $y \in \partial \Omega \cap B_r(w)$ such that $B \equiv B_{\varrho/M}(a_{\varrho}(y)) \subset B_r(w) \cap \Omega$ for some $\varrho < r$ then
	\begin{align*}
		u(x,t) \geq C \Lambda \varrho, \qquad (x,t) \in 2^{-1}B \times [t_1,t_2]
	\end{align*}
	for some constant $C$ not depending on $u$. Let us now consider the rescaled function $v$ as follows
	\begin{align*}
		w(x,t) = \frac{1}{C \Lambda \varrho} u \left (y + \frac{\varrho}{2}x, t_1 + (C \Lambda \varrho)^{2-p} \left (\frac{\varrho}{2}\right )^p t \right )
	\end{align*}
	then $w$ is a solution in $B_2(0) \times [0,T]$ (where $T = (C \Lambda \varrho)^{p-2} \left (\frac{\varrho}{2}\right )^{-p} (t_2-t_1)$) such that
	\begin{align*}
		w(x,t) \geq 1, \qquad (x,t) \in B_{1/(4M)}(a_1(0)).
	\end{align*}
	Before using \cref{MuLowerSimple} we need to know that $T \geq T_0$ (where $T_0$ is from \cref{MuLowerSimple}), first note that
	\begin{align*}
		(C \Lambda \varrho)^{p-2} \left (\frac{2}{\varrho}\right )^p = (C \Lambda)^{p-2}  2^{p} \varrho^{-2}
	\end{align*}
	thus taking $\varrho$ small enough depending on $t_2-t_1$ and $\Lambda$ we get
	\begin{align*}
		(C \Lambda)^{p-2}  2^{p} \varrho^{-2} (t_2-t_1) \geq T_0.
	\end{align*}
	Now using \cref{MuLowerSimple,meas order} we get that
	\begin{align*}
		\mu_{w}(B_2(0) \times (0,T_0)) \geq 1/C,
	\end{align*}
	for a constant $C(p,n,M) > 1$. Scaling back to our original variables we obtain that
	\begin{align*}
		\frac{\mu(B_{\varrho}(y) \times [t_1,t_2])}{\varrho^{n+1}} \geq C \Lambda
	\end{align*}
\end{proof}
\begin{remark}
	What we can learn from the above proof is that if $t_2-t_1 \approx \varrho^2$ then the measure of a parabolic cylinder (heat equation $(\varrho,\varrho^2)$) is of size $\varrho^{n+1}$, which is exactly the same as for the caloric measure related to the heat equation. In this sense the non-degeneracy assumption that $u \geq d(x,\partial \Omega)$ `linearizes' the equation.
\end{remark}
The next lemma is essentially trivial in its conclusion given continuity of the gradient and the representation of the Riesz measure as the limit of $|\grad u|^{p-1}$ on the boundary, however the proof highlights a way of thought which would be important when moving to other domains. Moreover since we are assuming an NTA-domain with interior ball condition this proof is considering the circumstances fairly straight forward.
\begin{lemma}
	Let $\Omega \subset \R^n$ be a bounded NTA-domain satisfying the interior ball condition. Furthermore assume that in $Q = B_r(x_0) \cap \Omega \times (0,T)$ the measure $\mu$ vanishes. Then for any $w \in B_r(x_0) \cap \partial \Omega$ we have
	\begin{align*}
		(0,T) \in \ddegg(w).
	\end{align*}
\end{lemma}
\begin{proof}
	Assume that there is a point $w \in B_r(x_0) \cap \partial \Omega$ such that there is a time $t \in (0,T)$ for which
	\begin{align*}
		t \in \dpar(w).
	\end{align*}
	From this we can conclude that as in previous estimates (see proof of \cref{thm:Riesz_support}) that if we wish to connect a point $a_\varrho(w)$ using a forward Harnack chain (see \cref{NTA FChain}) the waiting time will be of order $\varrho^2$. This implies via a barrier argument as in the proof of \cref{thm linearization} that for any $\hat t > t$ there is a neighborhood of $w$ for which $\hat t \in \dpar(\cdot)$ which implies via \cref{thm:Riesz_support} that the measure $\mu \neq 0$ and thus we have a contradiction.
\end{proof}
\section{What happens in non-smooth domains?}
\label{sec:cone}
In an NTA-domain $\Omega \subset \R^n$ that satisfies the interior ball condition we can use a barrier function as in \cref{lower bound1} to obtain that given a solutions initial data, if the existence time is large enough we will eventually get a linearization effect, i.e. the set $\dpar \neq \emptyset$. In this section we provide an adaptation of a proof by Vázquez to the $p$-parabolic equation which proves that if we are in a conical domain there is a solution for which the support never reaches the tip of the cone, this implies that we will be in $\ddeg$ no matter how long the solution exists. 
\subsection{No support in a cone}
Let $S$ be a open connected subset of the $n-1$ dimensional unit sphere $S^{n-1}$ with a smooth boundary, and define the conical domain
\begin{align*}
	K(R,S) = \{r \sigma\, |\, \sigma \in S, 0 < r  R \},
\end{align*}
we are looking for non-negative solutions vanishing on the lateral surface of the cone,
\begin{align*}
	\Sigma(R,S)\times (0,T) = \{r \sigma\, |\, \sigma \in \partial S, 0 \leq r \leq R \} \times (0,T).
\end{align*}

The following argument is taken from \cite[p. 344-345]{Vaz} with exponents adapted to the parabolic $p$-Laplace equation.

To begin the argument we first need to define some similarity transforms and the scaling properties of the support. The similarity transform is given as
\begin{equation*}
	u_{\lambda}(x,t) = (T_{\lambda} u)(x,t) := \lambda^{q}u(x/\lambda,\mu t).
\end{equation*}
Let $f : S \to \R$ be a non-negative function on the spherical cap $S$ such that $f = 0$ on $\partial S$, then the semi-radial function $U(x) = |x|^q f(x/|x|)$ is independent of this transform. The $p$-parabolic equation is invariant under $T_\lambda$ if $\mu = \lambda^{q(p-2)-p}$ in the following
\begin{equation*}
	(u_\lambda)_t - \Delta_p u_\lambda = \lambda^{q}\mu u_t - \lambda^{q(p-1)-p} \Delta_p u.
\end{equation*}

Let $S_u(t)$ denote the support of a function $u$ at time $t$. Then $S_{u_\lambda}(t) = \lambda S_u(\mu t)$.
Define the distance to the origin from the set $S_u$
\begin{equation*}
	r_u(t) = \inf \{|x|: x \in S_u(t)\}
\end{equation*}
which satisfies the scaling property $r_{u_\lambda}(t) = \lambda r_u(\mu t)$.

Let now $U(x)$ be a function that we will use as initial data, and let us construct a solution which has zero initial value close to the origin and coincides with $U(x)$ outside a ball of size 1.
Specifically we will define $\bar u$ as a solution to the $p$-parabolic equation satisfying the initial data as follows
\begin{align*}
	\bar u(x,t)= U(x) \text{ for $|x| \geq 1$ and $\bar u(x,t) = 0$ for $0 < |x| < 1$, $(x,t) \in \partial_p \Omega_\infty$.}
\end{align*}
Let $a < 1$ be a given number, then due to the finite propagation there exists a $\tau > 0$ such that $S_{\bar u}(\tau) \cap B_a = \emptyset$. Furthermore since $U$ is a solution to the $p$-Laplace equation we have by the comparison principle that
\begin{align*}
	u(x,\tau) \leq U(x).
\end{align*}
Using the similarity transform with $\lambda = a$ produces a solution $\overline u_1 = T_a \overline u$ such that $\overline u_1 \leq U(x)$ and $\overline u_1(x,0) = 0$ iff $|x| < a$. We see that $\bar u_1$ satisfies
\begin{align*}
	\bar u_1(x,0)= U(x) \text{ for $|x| \geq a$ and $\bar u(x,t) = 0$ for $0 < |x| < a$, $(x,t) \in \partial_p \Omega_\infty$},
\end{align*}
which implies that $\bar u_1(x,0) \geq \bar u(x,\tau)$ and thus by the comparison principle we get
\begin{align*}
	\bar u_1(x,t) \geq \bar u(x,t+\tau)
\end{align*}
for all $t > 0$ which gives the following inequality concerning the supports
\begin{align*}
	r_{\bar u}(t+\tau) \geq r_{\bar u_1}(t) = a r_{\bar u}(\mu t).
\end{align*}
We now use the above to get for $t = \tau/\mu$
\begin{align*}
	r_{\bar u}\left (\frac{\tau}{\mu} + \tau \right ) \geq a r_{\bar u}\left(\mu \frac{\tau}{\mu} \right ) = a r_{\bar u}(\tau) \geq a^2,
\end{align*}
then an iteration yields
\begin{align*}
	r_{\bar u}(t_k) \geq a^{k+1} \quad \text{for} \quad t_k = \tau \sum_{j=0}^k \mu^{-j} = \tau \sum_{j=0}^{k} \left [\frac{1}{a} \right ]^{(q(p-2)-p)j}.
\end{align*}
We see that if $q(p-2)-p > 0$ then $t_k \to \infty$, i.e. $q > \frac{p}{p-2}$.

In conclusion we can say that if $q > \frac{p}{p-2}$ then the support of $u$ will never reach the vertex of the cone.
\section{Stationary interfaces}
\label{sec:interface} 
As we mentioned in \cref{rem:cont}, \cref{thm:Riesz_no_support} has consequences for stationarity of the `interface', i.e. the boundary of the support of the solution is stationary in time. The way we will illustrate this phenomenon is via a simple example
\begin{equation} \label{eq:numericalproblem}
	\left \{
	\begin{array}{rcll}
		u_t - \Delta_p u &=& 0, &\text{ in $(-1,1) \times (0,T)$ }\\
		v &=& 0, &\text{ on $\{-1\} \times (0,T)$}\\
		v &=& 1, &\text{ on $\{1\} \times (0,T)$ } \\
		v(x,0) &=& x_+^\frac{p}{p-2},&\text{ for $x \in [-1,1]$}.
	\end{array}
	\right .
\end{equation}
A consequence of \cref{rem:cont} is that the boundary of the support of $u$ in \cref{eq:numericalproblem} will be stationary for small times, furthermore if we let $\hat t$ be the critical time when the support starts moving, we can apply the barrier in the proof of \cref{thm memory} to obtain a lower bound on $\hat t$
\begin{align} \label{eq:conjlb}
	(c_p c_0)^{p-2} \leq \hat t.
\end{align}
In the above the constants $c_0,c_p$ are the ones in the proof of \cref{thm memory}, specifically
\begin{align} \label{eq:conjconst}
	(c_p c_0)^{p-2} = \left ( \frac{p-2}{p}\right )^{p-1} \frac{1}{n(p-2)+p}.
\end{align}
\begin{conjecture} \label{conj}
	The critical time $\hat t = (c_p c_0)^{p-2}$ is such that
	\begin{align*}
		u(0,\hat t) = 0, \quad u(0,t) > 0, \quad t > \hat t.
	\end{align*}
	for the solution $u$ in \cref{eq:numericalproblem}.
\end{conjecture}

We will explore the contents of \cref{conj} in this section, firstly in \cref{sec:upperbound} we provide a bound for $\hat t$ from below, secondly in \cref{sec:numeric} we perform a numerical simulation of the solution in \cref{eq:numericalproblem} to estimate $\hat t$.
\subsection{An upper bound on $\hat t$ in \cref{conj}}
\label{sec:upperbound}
To show an upper bound for $\hat t$ we will be building a sequence of barriers from below based on rescalings of the Barenblatt solution. 

\begin{theorem} \label{thm:conjub}
	The critical time $\hat t$ for which
	\begin{align*}
		u(0,\hat t) = 0, \quad u(0,t) > 0, \quad t > \hat t,
	\end{align*}
	holds for the solution $u$ in \cref{eq:numericalproblem}, satisfies
	\begin{align*}
		\hat t \leq \left ( \frac{p-2}{p}\right )^{p-1} \frac{1}{p}.
	\end{align*}
\end{theorem}

\begin{remark}
	Denote $t_+ = \left ( \frac{p-2}{p}\right )^{p-1} \frac{1}{p}$ the upper bound from the theorem above, and $t_-$ the lower bound from \cref{eq:conjlb} then for $n=1$ we have
	\begin{align*}
		\frac{t_+}{t_-} = 2 \frac{p-1}{p} > 1, \qquad p > 2.
	\end{align*}
	Unfortunately the upper bound in \cref{thm:conjub} does not depend on $n$, but we have captured the right asymptotic behavior for $p \to 2$.
\end{remark}

\begin{proof}
	To begin with our construction we first find the time where the Barenblatt solution has support $B_{1-\delta}$, $\delta \in (0,1)$. We assume that $C_0 = q$ (changes only the mass of the solution) and do the following computation
	\begin{align*}
		t^{-k} \left ( q - q \left ( \frac{1-\delta}{t^{k/n}}\right )^{\frac{p}{p-1}} \right )^{\frac{p-1}{p-2}} = 0
	\end{align*}
	we get the value of $t$ to be
	\begin{align*}
		(1-\delta)^\frac{n}{k} \equiv t_1(\delta).
	\end{align*}
	At $t_1$ the value of $\U(0,t_1)$ becomes
	\begin{align*}
		\U(0,t_1)=(1-\delta)^{-n} q^{\frac{p-1}{p-2}}.
	\end{align*}
	To allow us the flexibility we need in the following argument we set
	\begin{align*}
		\epsilon \lambda \equiv \U(0,t_1),
	\end{align*}
	for $\epsilon \in (0,1)$ to be chosen depending on $\delta$.
	Now consider the rescaled Barenblatt solution (still a solution to \cref{Hu} due to intrinsic scaling)
	\begin{align*}
		\hat \U(x,t) = \frac{1}{\lambda} \U(x,\lambda^{2-p} t+t_1)
	\end{align*}
	which at $t = 0$ is
	\begin{align*}
		\hat \U(x,0) = \frac{1}{\lambda} (1-\delta)^{-n} \left ( q - q \left ( \frac{|x|}{1-\delta}\right )^{\frac{p}{p-1}} \right )_+^{\frac{p-1}{p-2}} 
	\end{align*}
	\begin{align*}
		k = \left ( p-2 + \frac{p}{n} \right )^{-1}, \qquad q = \frac{p-2}{p} \left ( \frac{k}{n}\right )^{\frac{1}{p-1}}.
	\end{align*}
	Next, let us assume that $\hat U(x-1,0) \leq u(x,0)$ for $u$ as in \cref{eq:numericalproblem}, and let us find $t_2(\delta)$ such that the support of $\hat \U$ is $B_1$. This implies using the parabolic comparison principle that after $t_2$ the support of $u$ has moved away from $\{x = 0\}$.

	To proceed we need to choose $\epsilon$ given the value of $\delta$ such that $\hat U(x-1,0) \leq u(x,0)$ and is the unique largest value for which this inequality holds true. To find this $\epsilon$, we note that we wish to solve $\hat U(x,0) = |1-x|^{\frac{p}{p-2}}$ for a unique pair $x,\epsilon \in (0,1)$, i.e. we wish to solve
	\begin{align*}
		\frac{1}{\lambda^{p-2}} (1-\delta)^{-n} \left ( q - q \left ( \frac{r}{1-\delta}\right )^{\frac{p}{p-1}} \right )_+^{\frac{p-1}{p-2}} = (1-r)^{\frac{p}{p-2}}.
	\end{align*}
	Which by some manipulation yields the following
	\begin{align}\label{eq:tomte}
		\epsilon^{\frac{p-2}{p-1}} \left ( 1 - \left ( \frac{r}{1-\delta}\right )^{\frac{p}{p-1}} \right ) = (1-r)^{\frac{p}{p-1}}.
	\end{align}
	We wish to find a value for $\epsilon$ and $r$ such that the left hand side equals the right hand side, but the left hand side being smaller than the right hand side for all other values of $r$. This implies that at the point of contact their derivatives match and we can thus consider the simplified equation of the $r$ derivative of \cref{eq:tomte} (which has a solution for any $\epsilon$)
	\begin{align*}
		-\epsilon^{\frac{p-2}{p-1}} \frac{p}{p-1} \frac{1}{1-\delta} \left ( \frac{r}{1-\delta}\right )^{\frac{1}{p-1}} = - \frac{p}{p-1} (1-r)^{\frac{1}{p-1}}
	\end{align*}
	some algebraic manipulations later and we arrive at
	\begin{align} \label{eq:r_intersect}
		r  = \frac{(1-\delta)^{p}}{\epsilon^{p-2}+(1-\delta)^{p}}.
	\end{align}
	Plugging the value of $r$ from \cref{eq:r_intersect} into \cref{eq:tomte} gives us the problem of solving
	\begin{align*}
		\epsilon^{\frac{p-2}{p-1}} \left ( 1 - \left ( \frac{(1-\delta)^{p-1}}{\epsilon^{p-2}+(1-\delta)^{p}}\right )^{\frac{p}{p-1}} \right ) = \left (1-\frac{(1-\delta)^{p}}{\epsilon^{p-2}+(1-\delta)^{p}} \right )^{\frac{p}{p-1}}.
	\end{align*}
	As can be shown by a tedious calculation, the above equation is equivalent to the following
	\begin{align*}
		\epsilon^{\frac{p-2}{p-1}} \left ( 1- \frac{1}{((1-\delta)^p + \epsilon^{p-2})^{\frac{1}{p-1}}}\right ) = 0,
	\end{align*}
	which is solved by
	\begin{align} \label{eq:e_intersect}
		\epsilon = \left (1 - (1-\delta)^p \right )^\frac{1}{p-2}.
	\end{align}
	With the values of \cref{eq:r_intersect,eq:e_intersect} we can calculate the value of $t_2$ for the function $\hat \U$. Considering the definition of $\hat \U$ we see that $t_2$ satisfies the following
	\begin{align*}
		t_2(\delta) &= \lambda^{p-2}(1-t_1) = (1-\delta)^{-n(p-2)} q^{p-1} \epsilon^{2-p} (1-(1-\delta)^{n/k}) \\
		&=\frac{q^{p-1}(1-(1-\delta)^{n/k})}{(1-\delta)^{n(p-2)} (1 - (1-\delta)^p  )}
	\end{align*}
	which when $\delta \to 0$ becomes (a lengthy calculation shows that $t_2(\delta)$ is decreasing as $\delta \to 0$)
	\begin{align*}
		t_2(0) = q^{p-1}\frac{n}{pk} = \left(\frac{p-2}{p} \right )^{p-1} \frac{k}{n} \frac{n}{pk} = \frac{(p-2)^{p-1}}{p^p}.
	\end{align*} 
\end{proof}

\subsection{Numerical experiment of \cref{conj}}
\label{sec:numeric}

We will be using MOL (Method of Lines) to solve \cref{eq:numericalproblem}, which amounts to discretizing the equation in space leaving us with a system of non-linear ODE's (semi-discretization). To specify our numerical setup, consider the spatial discretization with $N$ steps and $dx \approx 1/N$, then the MOL equation becomes in a finite difference (FD) context
\begin{align} \label{eq:MOL}
	u_t^i = (p-1) (D^i_j u^j)^{p-2} H^i_j u^j, \qquad i = 0,\ldots, N,
\end{align}
where
\begin{align*}
	D^i_j u^j = 
	\begin{cases}
		\frac{u^{i+1} - u^{i-1}}{2 dx}, &\text{ if } 1 \leq i < N\\
		0
	\end{cases}
\end{align*}
is a basic FD type difference quotient (central difference quotients) and
\begin{align*}
	D^i_j u^j = 
	\begin{cases}
		(u^{i+1}-2u^i + u^{i-1})/(dx^2), &\text{ if } 1 \leq i < N\\
		0
	\end{cases}
\end{align*}
is a basic second order FD difference quotient. Thus we see that \cref{eq:MOL} is a system of $N+1$ nonlinear ODE's. The system \cref{eq:MOL} turns out to be stiff and sparse, we will be using Matlab's stiff solver \emph{ode15s} to solve this system numerically with the data given as in \cref{eq:numericalproblem}, with $T = 1.2 \hat t$ ($\hat t$ from \cref{conj}). See \cite{J} for convergence of a FEM semi-discretization.
\begin{figure} 
	\begin{center}
		\includegraphics[width = 12cm]{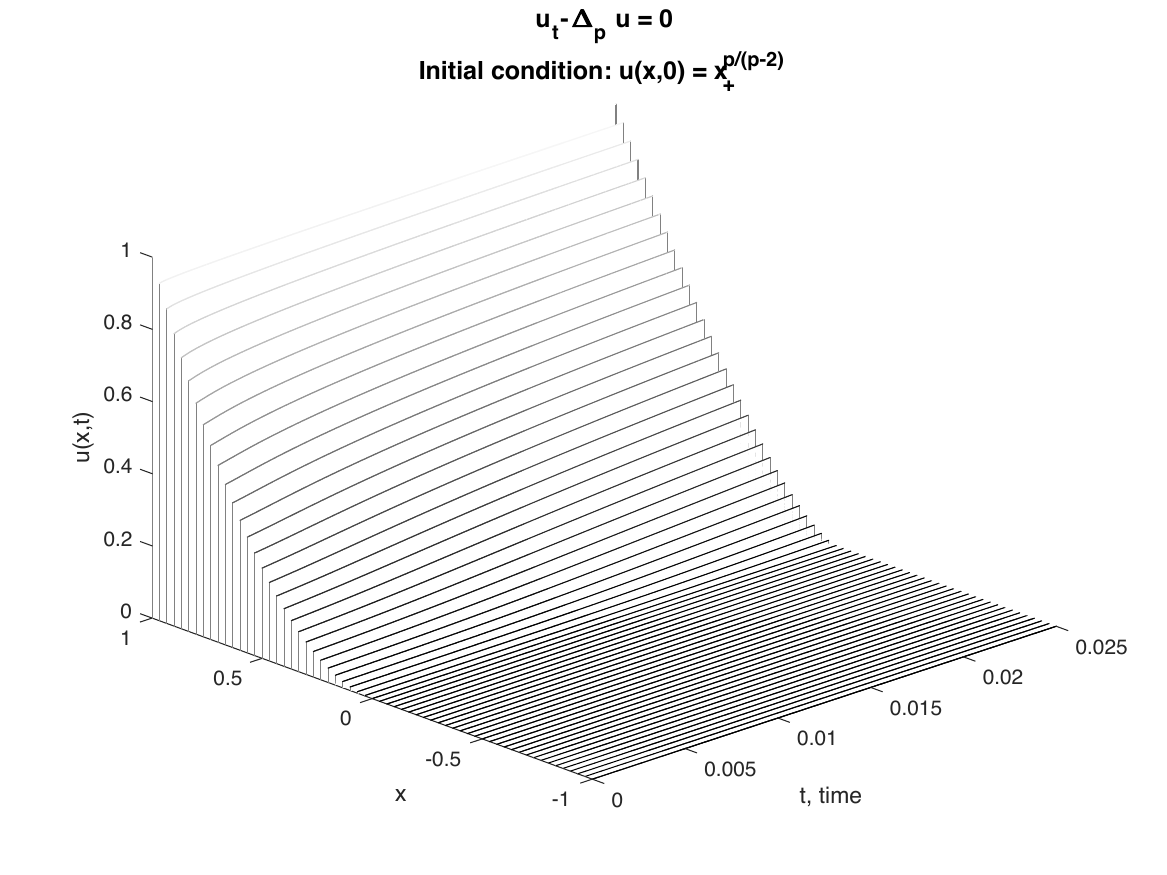}
	\end{center}
	\caption{Waterfall plot of the numerical solution of \cref{eq:MOL}, $p = 4$, observe the difference in profile at beginning and end.}
	\label{fig:MOL}
\end{figure}

Since we are interested in the support of our numerical solution, we will consider the equation satisfied by the interface ($u \equiv 0$), actually we will be considering the equation of a basic level set. That is, we are looking for a curve $\gamma : \R_+ \to \R$ such that for a given level $M \geq 0$ the following holds
\begin{align*}
	u(\gamma(t),t) = M.
\end{align*}
We will be considering the above problem with the initial point $\gamma(0)$ be be a point in the initial data that is equal to $M$, and for $M = 0$ it will be the edge of the support, i.e. $x = 0$. Proceeding formally and differentiating, the condition for $\gamma$ gives us,
\begin{align} \label{eq:levelseteq}
	u_x \gamma' + u_t = 0,
\end{align}
which after inserting the equation \cref{Hu} into \cref{eq:levelseteq} yields
\begin{align*}
	-\gamma'(t)  = \frac{\Delta_p u}{u_x}(\gamma(t),t) = (p-1) (u_x(\gamma(t),t))^{p-3} u_{xx}(\gamma(t),t).
\end{align*}
In \cref{fig:Interface} we see the result of the above equation when using the numerical solution of \cref{eq:MOL} seen in \cref{fig:MOL} as the approximate values for $u$ and approximating the first and second derivative with the respective FD quotients, as described in \cref{eq:MOL}.
\begin{figure}
	\begin{center}
		\includegraphics[width = 12cm]{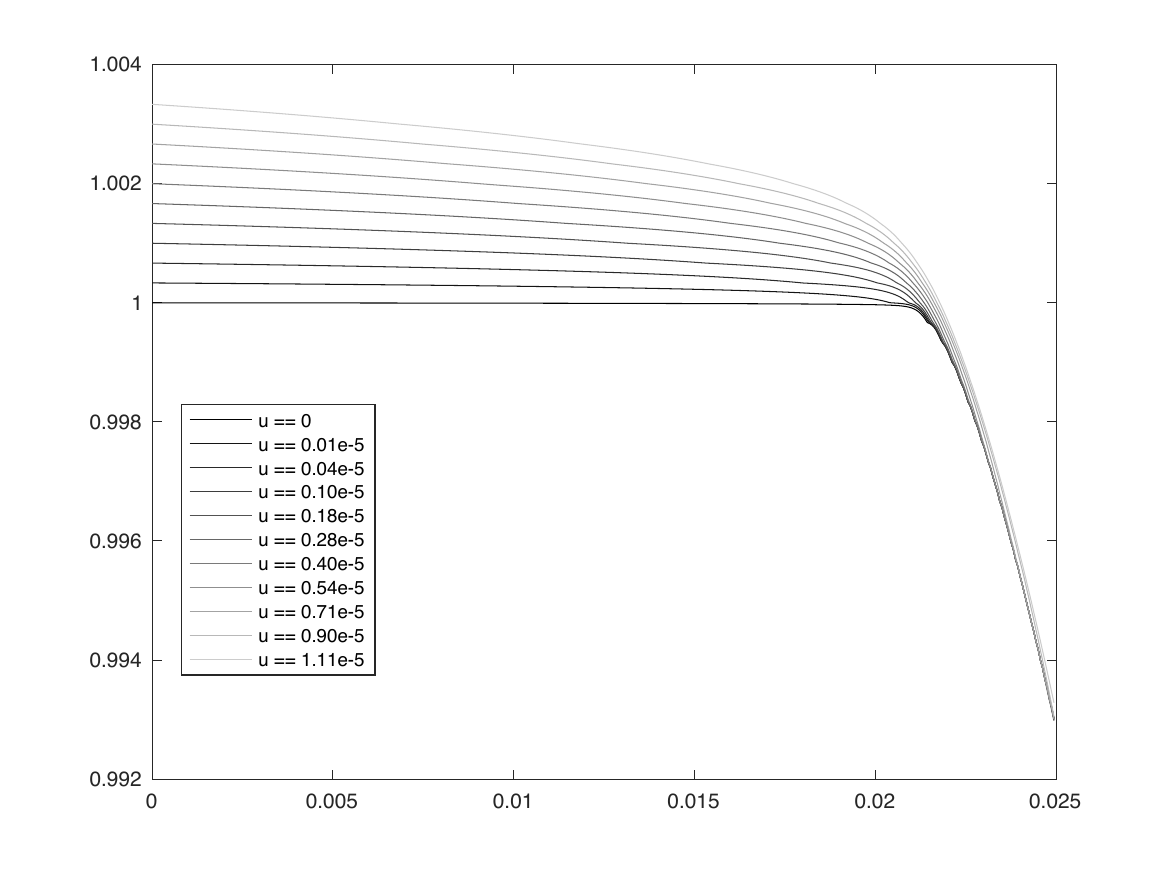}
	\end{center}
	\caption{Level sets, conjectured critical time is $t \approx 0.0208$, $p = 4$.}
	\label{fig:Interface}
\end{figure}
Upon visual inspection of \cref{fig:Interface} we see the sharp deviation of the support after roughly $t = 0.02$, which coincides quite well with the conjectured value of $0.0208$ for $p = 4$.

Based upon heuristic ideas we can expect that the profile of the solution towards the edge of the support will after the critical time behave like the Barenblatt solution, i.e. $u(x) \approx d(x,S(t)^C)^\frac{p-1}{p-2}$, where $S(t)$ is the support of the solution at $t$. The numerical result of the behavior at $1.2 \hat t$ can be found in \cref{fig:Profile} and the coincidence is striking, the edge of the support is estimated using the solution in \cref{fig:Interface}.

\begin{figure} 
	\begin{center}
		\includegraphics[width = 12cm]{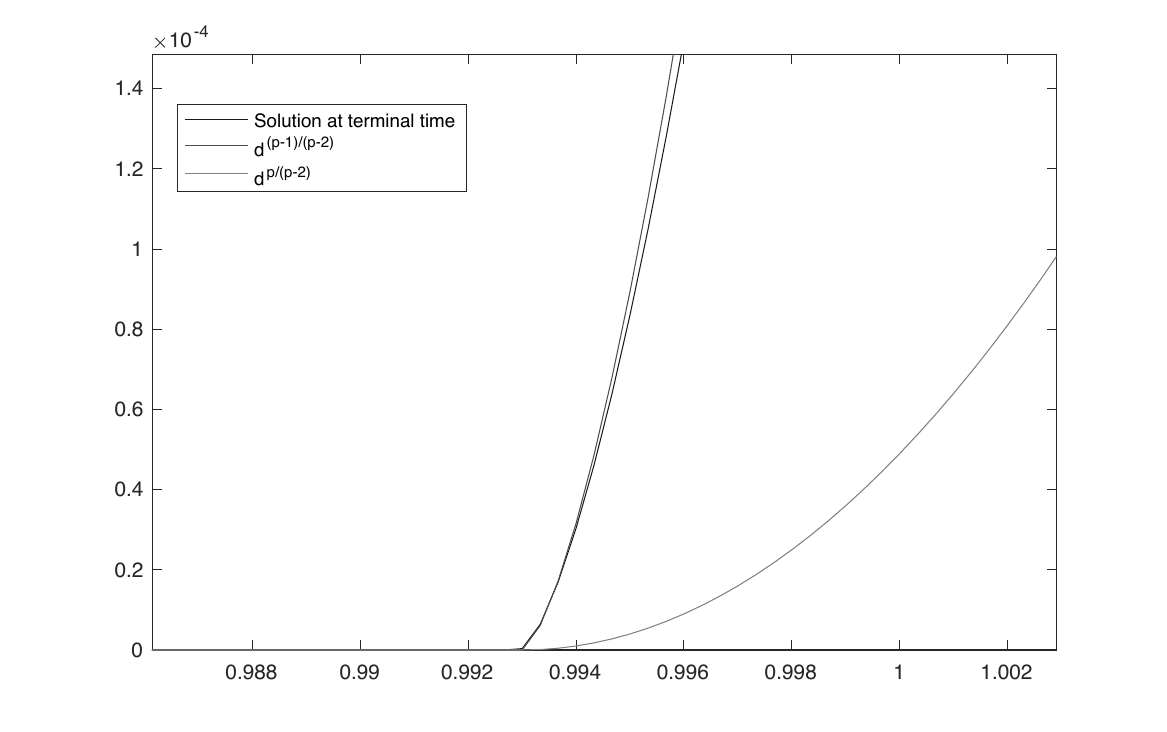}
	\end{center}
	\caption{Plot of the numerical solution of \cref{eq:MOL}, $p = 4$ at time $1.2 \hat t$, observe the coincidence of the profile of the solution with the profile of the Barenblatt, as compared to the profile of the initial datum (shifted to become zero at the same point as the other functions).}
	\label{fig:Profile}
\end{figure}


\newpage
\begin{thebibliography}
	{AAAA}
	
	\bibitem{A} B. Avelin, On time dependent domains for the degenerate $p$-parabolic equation: Carleson estimate and H{\"o}lder continuity, \emph{Math. Ann.} \textbf{364}(1) (2016), 667--686.
	
	\bibitem{AGS} B. Avelin, U. Gianazza and S. Salsa, Boundary estimates for certain degenerate and singular parabolic equations, \emph{J. Eur. Math. Soc.}, \textbf{18}(2) (2016), 381--426.
	
	\bibitem{AKN} B. Avelin, T. Kuusi and K. Nyström, Boundary behavior of solutions to the\\ parabolic $p$-Laplace equation, \emph{Anal. PDE}, \textbf{12}(1) (2019), 1–-42.
	
	\bibitem{BBG} A. Bj\"orn, J. Bj\"orn and U. Gianazza, The Petrovskiĭ criterion and barriers for degenerate and singular p-parabolic equations, \emph{U. Math. Ann}, \textbf{368}(3--4) (2017), 885--904.
	
	\bibitem{BBGP} A. Bj\"orn, J. Bj\"orn, U. Gianazza and M. Parviainen, Boundary regularity for degenerate and singular parabolic equations, \emph{Calc. Var. Partial Differential Equations} \textbf{52}(3) (2015), 797--827.
	
	%\bibitem{HKM} J. Heinonen, T. Kilpeläinen, O. Martio, Nonlinear potential theory of degenerate elliptic equations. Unabridged republication of the 1993 original. Dover Publications, Inc., Mineola, NY, 2006. xii+404 pp.
	
	\bibitem{DB} E. DiBenedetto,  \emph{Degenerate parabolic equations}, Springer Verlag, Series Universitext, New York, (1993).
	
	\bibitem{GNL} U. Gianazza, N. Liao and T. Lukkari, A Boundary Estimate for Singular Parabolic Diffusion Equations, \emph{NoDEA Nonlinear Differential Equations Appl.} \textbf{25}(4) (2018). 
	
	\bibitem{GS} U. Gianazza and S. Salsa, On the boundary behaviour of solutions to parabolic equations of $p$−Laplacian type, \emph{Rend. Istit. Mat. Univ. Trieste} \textbf{48} (2016), 463–-483.
	
	\bibitem{JK} D. Jerison and C. Kenig, Boundary behavior of harmonic functions in non-tangentially accessible domains, \emph{Adv. Math.} \textbf{46} (1982), 80--147.
	
	\bibitem{J} N. Ju, Numerical Analysis of Parabolic p-Laplacian: Approximation of Trajectories, \emph{SIAM Journal on Numerical Analysis} \textbf{37}(6) (2000), 1861--1884.
	
	\bibitem{KL} T. Kilpel{\"a}inen and P. Lindqvist,   On the Dirichlet boundary value problem for a degenerate parabolic equation., \emph{ SIAM J. Math. Anal.} \textbf{27}(3) (1996), 661--683.
	
	\bibitem{KMN} T. Kuusi, G. Mingione and K. Nystr{\"o}m, A boundary Harnack inequality for singular equations of $p$-parabolic type. \emph{Proc. Amer. Math. Soc.} \textbf{142}(8) (2014), 2705--2719.
	
	\bibitem{LN1} J. Lewis and K. Nystr{\"o}m, Boundary behavior for $p$-harmonic functions in Lipschitz and starlike Lipschitz ring domains. \emph{Ann. Sci. \'Ecole Norm. Sup. (4)} \textbf{40}(5) (2007), 765--813.
	
	\bibitem{LN2} J. Lewis and K. Nystr{\"o}m, Boundary behavior and the Martin boundary problem for $p$-harmonic functions in Lipschitz domains. \emph{Ann. of Math. (2)} \textbf{172}(3) (2010), 1907--1948.
	
	\bibitem{Vaz} J.L. Vázquez, The porous medium equation. Mathematical theory. Oxford Mathematical Monographs. The Clarendon Press, Oxford University Press, Oxford, 2007. xxii+624 pp. 
	
\end{thebibliography}
\end{document}